\documentclass[a4paper, 11pt, table]{amsart}
\usepackage{amsmath, amsthm, amscd, amssymb, amsfonts, amsxtra, amssymb, latexsym, mathtools}
\usepackage{enumerate}
\usepackage{verbatim}
\usepackage{textcomp}

\usepackage{multirow}
\usepackage[table]{xcolor}
\usepackage{graphicx, epsfig, tikz, pifont}
\usepackage[utf8x]{inputenc}

\usepackage{float}
\usepackage{pb-diagram}

\usepackage{tikz-cd}

\usetikzlibrary{arrows,%
                petri,%
                topaths}%
\usetikzlibrary{automata}                
\usetikzlibrary{positioning}

\usepackage[justification=centering]{caption}

\newcommand{\Z}{\mathbb{Z}}

\newcommand{\D}{\mathbb{D}}
\newcommand{\Q}{\mathbb{Q}}
\newcommand{\N}{\mathbb{N}}
\newcommand{\ff}{\mathbb{F}}
\newcommand{\F}{\mathbb{F}}

\newcommand{\Prep}{\mathcal{P}}

\newcommand{\sk}{\smallskip}

\newcommand{\Sym}{\mathbb{S}}

\newtheorem{thm}{Theorem}[section]
\newtheorem{prop}[thm]{Proposition}

\newtheorem{coro}[thm]{Corollary}
\theoremstyle{definition}
\newtheorem{rem}[thm]{Remark}
\newtheorem{exam}[thm]{Example}

\theoremstyle{remark}

\theoremstyle{definition}
\newtheorem{defi}[thm]{Definition}

\usepackage{color}

\usepackage{hyperref}
\hypersetup{colorlinks = true,	allcolors  = blue}

\begin{document} \sloppy
\numberwithin{equation}{section}
\title{Lee metrics on groups}
\author{Ricardo A.\@ Podest\'a, Maximiliano G.\@ Vides}
\dedicatory{\today} 
\keywords{Invariant metric, weight function, finite groups, symmetry groups}
\thanks{2010 {\it Mathematics Subject Classification.} Primary 54E35;\, Secondary 20F65, 94B05.}
\thanks{Partially supported by CONICET, FONCyT and SECyT-UNC}

\address{Ricardo A.\@ Podest\'a, FaMAF -- CIEM (CONICET), Universidad Nacional de C\'ordoba, \newline
 Av.\@ Medina Allende 2144, Ciudad Universitaria, (5000) C\'ordoba, Rep\'ublica Argentina. \newline
{\it E-mail: podesta@famaf.unc.edu.ar}}

\address{Maximiliano G.\@ Vides 
FaMAF -- CIEM (CONICET), Universidad Nacional de C\'ordoba, \newline
 Av.\@ Medina Allende 2144, Ciudad Universitaria, (5000)   C\' ordoba, Rep\'ublica Argentina. \newline
{\it E-mail: mvides@famaf.unc.edu.ar}}

\begin{abstract}
In this work we consider interval metrics on groups; that is, integral invariant metrics whose associated weight functions do not have gaps. We give conditions for a group to have and not to have interval metrics. Then we study Lee metrics on general groups, that is interval metrics having the finest unitary symmetric associated partition. These metrics generalize the classic Lee metric on cyclic groups. 
In the case that $G$ is a torsion-free group or a finite group of odd order, we prove that $G$ has a Lee metric if and only if $G$ is cyclic. Also, if $G$ is a group admitting Lee metrics then $G \times \Z_2^k$ always has Lee metrics for every $k \in \N$. Then, we
show that some families of metacyclic groups, such as cyclic, dihedral, and dicyclic groups, always have Lee metrics. 
Finally, we give conditions for non-cyclic groups such that they do not have Lee metrics. We end with tables of all groups of order $\le 31$ indicating which of them have (or have not) Lee metrics and why (not).
\end{abstract}

\maketitle

\section{Introduction}
The Lee metric, defined over the ring of integers $\Z_m$, was introduced in 1958 by Lee \cite{Lee58} as an alternative to the standard Hamming metric on $\ff_q^n$. Almost forty years later, Hammons et al.\@ \cite{HKCSS} showed that the binary nonlinear Preparata and Kerdock codes can be considered as the image under the Gray map $\mathcal{G}: \Z_4^n \rightarrow \ff_2^{2n}$ of a linear code over $\Z_4$ (see also the work of Nechaev \cite{Ne}). This gave a new impulse to the study of linear codes over $\Z_4$ (or $\Z_m$) with the Lee metric and more generally to codes over arbitrary rings such as Galois rings $G(p^m,d)$. 

Later, Carlet \cite{Ca} and Yildiz-Özger \cite{YO} defined more general Gray maps from $\Z_{p^k}$ to $\Z_p^r$, with $r=p^{k-1}$, which allowed to introduce certain generalizations of Lee metrics. In this case, the focus is on the property of having a weight preserving map from $\Z_{p^k}$ to a Hamming space $(\Z_p^{r},d_{Ham})$.
In \cite{PV}, we studied isometric embeddings and extensions of metrics. The metrics defined by Carlet and Yildiz-Özger are thus extended Lee metrics in the sense defined in \cite{PV}. We have obtained more \textit{extended Lee metrics} from $\Z_m$ to the Hamming space $(\Z_k^{{\small m/k}},d_{Ham})$.

In another direction, Batagelj \cite{Ba95} studied weight functions (norms) over groups, more specifically interval weights, and defined a (generalized) Lee weight. 
The classical Lee metric on cyclic groups has the property that among all integer valued metrics that take all the values from $0$ to $n$, it achieves the maximum possible value for $n$.
We will abstract this property to study the existence of Lee metrics on arbitrary groups.
In \cite{PV2} we studied invariant metrics on groups $G$ by considering unitary symmetric partitions of $G$ and noticed that the existence of a Lee metric for a group $G$ can be thought as a property of the finest of such partitions, that we called the Lee partition of a group $G$. Such Lee partition always exists, and as we proved in \cite[Lemma 2.4]{PV2}, every unitary and symmetric partition of a group $G$ gives rise to an invariant metric. Furthermore, it is possible to choose this metric to be integral. However, it may not always define an interval metric, as we shall see.
In the cases that the Lee partition defines an interval metric we will say that it is a Lee metric.

\subsubsection*{Invariant metrics on groups} 
A \textit{metric} on a set $X$ is a function $d : X \times X \rightarrow\mathbb{R}_{\geq 0}$ such that for every $x,y,z\in X$ satisfies the following three conditions:
\begin{enumerate}[($d1$)]
\setlength{\itemsep}{1mm}
\item $d(x,y)\geq 0$ and $d(x,y)=0 \,\Leftrightarrow\, x=y$ (positiveness);
\item $d(x,y)=d(y,x)$ (commutativity);
\item $d(x,y)\leq d(x,z) + d(x,z)$ (triangular inequality).
\end{enumerate}
One says that $(X,d)$ is a \textit{metric space}. 
If $X=G$ is a group, then we will say that $(G,d)$ is a \textit{metric group}. 
A \textit{weight} on $G$ is a function $w : G \rightarrow \mathbb{R}_{\ge 0}$ such that for every $x,y\in G$ satisfies:
\begin{enumerate}[($w1$)]
\setlength{\itemsep}{1mm}
  \item positiveness: $w(x) \geq 0$ and $w(x)=0$ if and only if $x=e$, the identity in $G$;
  \item symmetry: $w(x)=w(x^{-1})$;
  \item triangular inequality: $w(xy) \leq w(x) + w(y)$.
\end{enumerate}

We now recall the basic definitions and notations of invariant metrics on groups from Section~2 in \cite{PV2}. 
From now on we assume that $(G,d)$ is a metric group. 
The metric $d$ is called \textit{right} (resp.\@ \textit{left}) \textit{translation invariant} if for any $g,g',h$ in $G$ we have 
$$d(gh,g'h)=d(g,g')$$ 
(resp.\@ $d(hg,hg')=d(g,g')$). If $d$ is both right-invariant and left-invariant we say that $d$ is \textit{bi-invariant}.
For $G$ abelian both notions coincide and $d$ is bi-invariant.   

Given $(G,d)$ a metric group we have the induced weight function $w : G\rightarrow \mathbb{R}_{\ge 0} $ given by  
$w(x)=d(x,e)$ for any $x \in G$, where $e$ is the identity element of $G$. 
Conversely, if $(G,w)$ is a weight space, one can define a metric $d$ on $G$ by 
$d(x,y) = w(xy^{-1})$ 
for every $x,y \in G$ provided that $w(x^{-1})=w(x)$ for every $x \in G$ (denoted $d(x,y) = w(x-y)$ and $w(-x)=w(x)$ if $G$ is abelian). This metric is right translation invariant by definition. 
Bi-invariant metrics are closely related with conjugacy classes.
If $d$ is bi-invariant then $w$ is constant on conjugacy classes. That is, for every $x,y \in G$ we have 
$$w(x)=d(x,e)=d(yx,y)=d(yxy^{-1}, yy^{-1})=d(yxy^{-1},e)=w(yxy^{-1})$$
where in the second and third equalities we have used left and right invariance, respectively.	
Conversely, if $d$ is right (or left) invariant	and $w$ is constant on conjugacy classes then $d$ is bi-invariant.
In fact, for every $x,y,z \in G$ we have 
$$ d(zx,zy)=w(zx(zy)^{-1})=w(zxy^{-1}z^{-1})=w(xy^{-1})=d(x,y).$$

Let $(G,d)$ be a metric group (recall that $d$ is right translation invariant) with associated weight function $w$. 
The \textit{induced partition} of $(G,d)$, denoted by $\mathcal{P}=P(G,d)$, is the partition of $G$ determined by the equivalence relation
\begin{equation} \label{eq part}
g \sim h \quad \Leftrightarrow \quad w(g)=w(h) \quad \Leftrightarrow \quad d(g,e)=d(h,e) 
\end{equation} 
for any $g,h \in G$.                               
We will denote by $P_0, \ldots, P_s$ the parts of $\mathcal{P}$, that is 
$$\mathcal{P}=\{P_0,\ldots,P_s\}.$$
Thus, if $w_0, w_1, \ldots, w_s$ are the different weights of $G$ (i.e.\@ the different real values that $w(g)$ can take for $g\in G$), we will say that $G$ is an \textit{$s$-weight} metric group. In this case we have that 
	$$P_i=\{g\in G : w(g)=w_i\}.$$
We will say that $\mathcal{P}$ is \textit{unitary} if $\{ e \} \in P$, in which case we assume that $P_0=\{e\}$, and that it is \textit{symmetric} if 
$$P_i=P_i^{-1},$$ 
i.e.\@ $g \in P_i$ if and only if $g^{-1} \in P_i$, for all $i=0,\ldots,s$. 
The partition $P_0,P_1,\ldots,P_s$ associated to a bi-invariant metric of $G$ is unitary symmetric and, by the previous comments, also \textit{conjugate}, i.e.\@ 
$gP_i g^{-1}=P_i$ for all $g\in G$ and $i=0,\ldots,s$.

A metric $d$ on $G$ is called \textit{integral} if it only takes integer values. 
Since our main motivation to study metrics on groups are applications to combinatorics (coding theory, for instance), 
from now on we will only consider integral metrics. Out of these, we will only focus on the so called interval and $\mathcal{P}$-interval metrics (see Definitions \ref{interval} and \ref{P-interval}) and on a special type of interval metrics, the Lee metrics, metrics generalizing the Lee metrics on cyclic groups.


\noindent 
\textit{Outline and results.}
Briefly, in Section 2 we consider interval metrics, in Section 3 we study Lee metrics and groups admitting them and in Section 4 we study groups not allowing Lee metrics. Finally, in Section 5 we study Lee metrics on finite groups of small order.

Now, we summarize the results in the paper in more detail.
Let $(G,d)$ be a metric group, where $d$ is an invariant metric, and let $\mathcal{P}=P(G,d)$ be the corresponding unitary symmetric partition with parts $P_0, P_1,\ldots,P_s$.

In Section 2 we study \textit{interval} metrics, that is integral invariant metrics without gaps in its weight values. 
In Proposition \ref{prop noIIM} we give a criterion on the parts $P_i$ of the partition $\mathcal{P}$ such that 
there is no interval metric $d$ associated to the partition $\Prep$. As a result, in Corollary~\ref{Zp no int} we show that the cyclic group $\Z_p$ has at least one non $\Prep$-interval metric for every prime $p$ with $p\equiv 1 \pmod 4$. On the other hand, in Proposition \ref{2-weights} we give a criterion which ensures that the group admits an interval metric for partitions with few parts ($2 \le s \le 4$).
Then, in Theorem~\ref{kGkH} we give conditions for a group $G$ having a subgroup $H$ of index 2 to have non $\mathcal{P}$-interval metrics. As a consequence, we get Corollaries \ref{ZnDnQnSn} and \ref{G>16}. In the first one we show that the cyclic groups $\Z_{2n}$ for $n\ge 7$, the dihedral groups $\D_n$ for $n\ge 4$, the dicyclic groups $\Q_{4n}$ for $n\ge 3$, the symmetric groups $\Sym_n$ for $n\ge 4$, the groups $\Z_2^r$ for any $r\ge 3$, and $G\times \Z_2^r$ for any non-trivial group $G$ and $r\ge 2$ have non $\mathcal{P}$-interval metrics. In the second one, we prove that if $|G|\ge 16$ and $G$ has a subgroup of index $2$ then $G$ has non $\mathcal{P}$-interval metrics. In particular, any abelian group of even order $2n$ with $n\ge 8$ has non $\mathcal{P}$-interval metrics.  Using the results in the section, in Examples~\ref{examG<8}--\ref{z3xz3} and Remark \ref{remarquito} we give a complete study of the existence of $\mathcal{P}$-interval metrics for all the groups of order $\le 9$.

In the next two sections we study \textit{Lee metrics} on an arbitrary group $G$, that is interval metrics whose associated unitary symmetric partitions are the finest ones. 
In Section 3 we introduce Lee metrics on groups (not necessarily cyclic) and study groups \textsl{admitting} these metrics. 
We first obtain general results. In Proposition \ref{some Lee} we prove that 
if $G$ has a Lee metric then $G \times \Z_2^{k}$ has a Lee metric for any $k\in \N$ and in Theorem \ref{odd order}
we show that if $G$ is finite of odd order, then $G$ has a Lee metric if and only if $G$ is cyclic.
In Proposition \ref{charact bi}, we characterize all the groups having bi-invariant Lee metrics, that is, those groups $G$ in which any given (invariant) Lee metric $d_{Lee}$ is also bi-invariant. It turns out that $d_{Lee}$ is bi-invariant if and only if $G$ is abelian or $G= \Q_8 \times \Z_2^k$ for some $k\in \N_0$.
Then we study some particular families of metacyclic groups. We show in Theorems \ref{Dn Lee} and \ref{Qn Lee} 
that dihedral groups $\D_n$ and dicyclic groups $\Q_{4n}$ (including the generalized quaternion groups) 
always have Lee metrics. 
In particular, by the above mentioned propositions, we obtain that the groups $G\times \Z_2^k$ with $G=\D_n, \Q_{4n}$ has Lee metrics for every $k \in \N$ and $n\ge 2$.

In Section 4 we study groups \textsl{not admittin}g Lee metrics. It is immediate from Theorem \ref{odd order} that non-abelian groups of odd order do not have Lee metrics. In Theorem \ref{non abelian cond} we generalize this result to non-abelian groups of any order.
In fact, we prove that if $G$ is non-cyclic with $Z(G)$ not having elements of order $2$ and the subgroup generated by 
the elements of order 2 and 4 is not the whole group, 
then $G$ has no Lee metrics. As a result, in Corollary \ref{torsion-free} we show that if $G$ is a torsion-free group, then $G$ has a Lee metric if and only if $G$ is cyclic. As another consequence of Theorem \ref{non abelian cond}, in Proposition \ref{aff}
we show that the affine group $\mathrm{Aff}(\ff_q)$ of the finite field of $q$ elements
has no Lee metrics for all $q\ne 2,3,5$.
Moreover, any subgroup of the form $\ff_q \rtimes H$ with $H$ a subgroup of $\ff_q^*$ where $H \not\simeq \{1\},\Z_2, \Z_4$, has no Lee metrics.
	
In the final section we study Lee metrics on finite groups of small order. 
In Tables \ref{tablita}--\ref{tablitab2} we list the 93 groups of order $\le 31$ indicating which of them have (not) Lee metrics and why (not).

\goodbreak

\section{Interval metrics on groups}
Here we will study a special type of integral metrics. 
We will say that an integral metric is \textit{gapless} or without gaps 
if and only if the associated weights are $0,1, \ldots,n,$ for a partition with $n$ non-trivial parts or else the weight function takes all the values in $\N_0$ if the partition has infinitely many parts.
We have the following definition.
\begin{defi}[\cite{Ba95}] \label{interval}
An \textit{interval metric} on a group is a gapless invariant metric with a finite number of weights. 
\end{defi}
In this case, we have (assuming $0=w_0 < w_1 < w_2 < \cdots < w_n$)  that 
$$P_i = \{ x \in G : w(x)=i \}.$$
For instance, this is the case for the Hamming and Lee metrics on $\Z_n$ with weights $0,1$ and $0,1,\ldots,\lceil \frac n2\rceil$, respectively. They are interval metrics with the minimum and maximum number of weights in $\Z_n$, respectively. 

Two metrics $d_1$ and $d_2$ on a group $G$ are said to be \textit{$\mathcal{P}$-equivalent}, denoted $d_1 \sim d_2$, if they define the
	same unitary symmetric partitions on $G$, i.e.\@ $P(G,d_1)=P(G,d_2)$.
By Lemma 2.5 in \cite{PV2} we know that any non-integral metric of $G$ with finite partition is 
$\mathcal{P}$-equivalent to some integral metric. However, it should be noted that the resulting metric is not necessarily an interval metric. Thus, we give the following definition. 
\begin{defi} \label{P-interval}
A metric $d$ on a group is called  \textit{$\mathcal{P}$-interval} if it is $\mathcal{P}$-equivalent to some 
interval metric $d'$.
\end{defi}

We will use the following notation: if $\{P_0, P_1,\ldots, P_s\}$ is a partition of a group $G$, for any $1 \le i \le j \le s$ we write 
\begin{equation} \label{PiPj}
P_i \ast P_j= P_i P_j \cup P_jP_i = \{gh : g\in P_i, h\in P_j\} \cup \{hg : g\in P_i, h\in P_j\}.
\end{equation}
Notice that if $j=i$, then $P_i \ast P_i = P_iP_i=\{gh : g,h \in P_i\}$. Of course, if $G$ is abelian we have that $P_i \ast P_j =P_iP_j=P_jP_i$ for ever $i\ne j$ and in this case we use the additive notation 
$P_i + P_j= \{g+h : g\in P_i, h\in P_j\}$.

\begin{rem} \label{Z13noIIM}
	We point out that not every metric is $\mathcal{P}$-interval.
	Consider the group $\Z_{13}$ with the partition $P_0=\{0\}$, $P_1=\{1,5,8,12\}$, $P_2=\{2,3,10,11\}$, $P_3=\{4,6,7,9\}$ and let $w$ be a weight function associated with this partition. 
	Notice that 
	$$P_i+P_i=P_0 \cup P_j \cup P_k$$ 
	for $i,j,k \in \{1,2,3\}$ with $i,j,k$ all different. Using this fact and the triangle inequality one can show by contradiction that any integral metric induced by this partition must have gaps (see the proof of the next proposition). 
	On the other hand, it is worth noticing that one can define an integral metric associated with the given partition by assigning the values $w_0=0$, $w_1=2$, $w_2=3$, and $w_3=4$, for example.
\end{rem}

Abstracting the idea in the previous remark we give a general criterion for a unitary symmetric partition not to admit a $\mathcal{P}$-interval metric. 
\begin{prop} \label{prop noIIM}
	Let $\mathcal{P}=\{P_0,P_1,\ldots,P_s\}$ be a unitary symmetric partition of a group $G$. If for every $i=1,\ldots,s$ we have 
	\begin{equation} \label{conditions} 
		P_i \ast P_i \subset P_0 \cup P_{j_{i,1}} \cup \cdots \cup P_{j_{i,\ell(i)}}
	\end{equation}
where $j_{i,1},j_{i,2},\ldots,j_{i,\ell(i)} \subset \{1,2,\ldots,s\}  \smallsetminus \{i\}$ 
and $(P_i \ast P_i) \cap P_{j_{i,k}} \ne \varnothing$ for all $k=1,\ldots, \ell(i)$ with $\ell(i) \ge 2$,
then every metric $d$ associated to the partition $\mathcal{P}$ of $G$ is not interval.
\end{prop}

\begin{proof}
	Let $w$ be the weight function on $G$ associated to $d$. 
	We now show that any integral metric induced by the partition $\mathcal{P}$ cannot be interval. 
	Suppose, by contradiction, that $w$ is interval and $w(P_i)=1$ for a fixed $i$ with $1\leq i \leq s$.
	By \eqref{conditions} and the assumptions following it, there are at least two elements $x\in P_j=P_{j_i,k_1}$, $y\in P_k=P_{j_i,k_2}$ with $j \ne k$ and $j,k \notin \{ 0,i\}$ that can be obtained as a product of $2$ elements in $P_i$, 
	say $x=g_1h_1$ and $y=g_2y_2$ with $g_1,g_2,h_1,h_2 \in P_i$. 
	By the triangle inequality, this means that $w(P_j) \le 2w(P_i) = 2$ and $w(P_k) \le 2w(P_i) =2$. 
	This implies that $w(P_j)=w(P_k)=2$, which is absurd since $w(P_j) \ne w(P_k)$.  
\end{proof}

We can now apply the proposition to generalize the counterexample given in Remark \ref{Z13noIIM}. 
\begin{coro} \label{Zp no int}
	For any prime $p\geq 13$ with $p\equiv 1 \pmod 4$ the cyclic group $\Z_p$ has at least one non $\mathcal{P}$-interval metric.
\end{coro}

\begin{proof}
	Let $s=\frac{p-1}4$ and let $P_1$ be the subgroup of the units of $\Z_p$ formed by the $s$-powers, i.e.\@ 
	$P_1=\{x^s : x \in \Z_p^*\}$.
	Notice that $P_1 \simeq \Z_4$ since, if $\alpha$ is a generator of the cyclic group $\Z_p^*$, then 
	$$P_1 = \{1,r,r^2,r^3\},$$ 
	where $r=\alpha^s$ has order $4$, $r^2=-1$ and $r^3=-r$. 

	Now, consider all the $s$ cosets $\{gP_1\}_{g\in \Z_p^*}$ of $P_1$ and denote them by $P_1, \ldots, P_s$. Hence, $\#P_i=4$ for all $i=1,\ldots,s$. 
	In this way, we have a unitary symmetric partition 
	$\mathcal{P}=\{P_0, P_1, \ldots,P_s\}$ of $\Z_p$, where $P_0=\{e\}$. 
	In fact, for each $i=1,\ldots,s$, we have 
		$$P_i=g_iP_1 =\{g_i, g_i r, g_i r^2, g_i r^3\} = \{g_i,g_i r,-g_i, -g_i r\}$$ 
	for some $g_i \in \Z_p^*$, and hence $P_i=P_i^{-1}$.
		
	We now look for conditions \eqref{conditions} in this case. 
	We will see that for every $i=1,\ldots,s$ we have 
\begin{equation} \label{Pijk}
	P_i + P_i  = P_0 \cup P_{2g} \cup P_{g(r+1)},
\end{equation}
	where $P_{\{ 2gr \}}, P_{\{ g(1+r)\}}\neq P_i$. Notice that it is enough to prove conditions \eqref{Pijk} for $i=1$, that is
	$$P_1 + P_1  = P_0 \cup P_{2} \cup P_{r+1}$$ 
	with $P_2, P_{r+1} \ne P_1$.
	First, observe that $2\not\in P_1$ since $2^4\equiv 1 \pmod p$ if and only if $p=3,5$, so $P_{2} \ne P_{1}$. Similarly, $1\not\in P_{r+1}$, since it would also lead to $2^4 \equiv 1 \pmod p$, hence  $P_{r+1}\neq P_{1}$.  Analogously, we can see that $2\not\in P_{g+1}$ since otherwise we would have $g=\pm 1$ or $3^4\equiv 1 \pmod p$, the latter only being possible if $p=2,5$, and hence $P_{2}\neq P_{r+1}$. 
	In particular, we have 
	$$(P_i+P_i) \cap P_{\{ 2g \}} \ne \varnothing \qquad \text{and} \qquad 
	(P_i+P_i) \cap P_{\{ g(r+1)\}} \ne \varnothing$$ 
	for every $i=1,\ldots, s$. 
	Thus, we are in the hypothesis of Proposition \ref{prop noIIM} and hence the weight $w$ has gaps, that is, $d$ is not interval.  
\end{proof}

Note that the partition of $\Z_{13}$ in Remark \ref{Z13noIIM} is given by the cosets of the cubes, that is 
$P_1=\{x^3:x\in \Z_{13}^*\}=\{1,5,8,12\}$, $P_2=2P_1=\{2,3,10,11\}$, and $P_3=4P_1=\{4,6,7,9\}$.

\begin{exam}
	By the corollary, $\Z_{17}$ and $\Z_{29}$ have some non $\mathcal{P}$-interval metrics. We now look at condition \eqref{Pijk} in more detail.
	
	$(i)$ For $\Z_{17}$ we have $s=4$ and 
	$P_1=\{x^4 : x \in \Z_{17}^*\}=\{1,4,13,16\}=\{\pm 1, \pm 4\}$ 
	and hence $P_2=2P_1=\{2,8,9,15\}=\{\pm 2, \pm 8\}$,
	$P_2=3P_1=\{3,5,12,14\}=\{\pm 3, \pm 5\}$, and $P_4=6P_1=\{6,7,10,11\}=\{\pm 6, \pm 7\}$. One can check that
	\begin{align*}
	P_1+P_1 = P_0 \cup P_2 \cup P_3, \qquad \qquad & P_3+P_3 = P_0 \cup P_2 \cup P_4,  \\
	P_2+P_2 = P_0 \cup P_1 \cup P_4, \qquad \qquad & P_4+P_4 = P_0 \cup P_1 \cup P_3,
	\end{align*}
and hence condition \eqref{Pijk} holds.
	
	\noindent $(ii)$
	For $\Z_{29}$ we have $s=7$ and 
	$P_1=\{x^7 : x \in \Z_{29}^*\}=\{1,12,17,28\}=\{\pm 1, \pm 12\}$ 
	and hence $P_2=2P_1=\{2,5,24,27\} = \{\pm 2, \pm 5\}$,
	$P_3=3P_1=\{3,7,22,26\}=\{\pm 3, \pm 7\}$, $P_4=4P_1=\{4,10,19,25\}=\{\pm 4, \pm 10\}$, $P_5=6P_1=\{6,14,15,23\}=\{\pm 6, \pm 14\}$, $P_6=8P_1=\{8,9,20,21\}=\{\pm 8, \pm 9\}$, and  $P_7=11P_1=\{11,13,16,18\}=\{\pm 11, \pm 13\}$. 
	One can check that
	\begin{align*}
	& P_1+P_1 = P_0 \cup P_2 \cup P_7, & P_2+P_2 = P_0 \cup P_3 \cup P_4,  && P_3+P_3 = P_0 \cup P_4 \cup P_5, \\
	& P_4+P_4 = P_0 \cup P_5 \cup P_6, & P_5+P_5 = P_0 \cup P_1 \cup P_6,  && P_6+P_6 = P_0 \cup P_1 \cup P_7, \\
	 					  &  & P_7+P_7 = P_0 \cup P_2 \cup P_3,  &&	 
	\end{align*}
and hence condition \eqref{Pijk} holds. 
	\hfill $\lozenge$
\end{exam}

For some metrics with few weights we can assure that they are $\mathcal{P}$-interval. For instance, every $2$-weight metric is $\mathcal{P}$-interval and every $3$-weight metric with a small part is $\mathcal{P}$-interval. We say that a part $P_i$ is \textit{small} if it is of the form 
$$P_i=\{g,g^{-1}\}$$
for some $g\in G$ (or $P_i=\{g\}$ if $g=g^{-1}$), i.e.\@ it is a smallest non-trivial symmetric part. 
Notice that $\Z_{13}$ with the $3$-weight metric of Remark~\ref{Z13noIIM} is not $\mathcal{P}$-interval and it has no small parts. 

\begin{prop} \label{2-weights}
	Let $G$ be a group with unitary symmetric partition $\Prep =\{P_0,P_1,\ldots,P_s\}$.
	If one of the following conditions hold:
	\begin{enumerate}[$(a)$]
		\item $s=2$; \sk
		
		\item $s=3$ and $P_i=\{g,g^{-1}\}$ for some $i\in \{1,2,3\}$; \sk 
		
		\item $s=4$, $P_i=\{g,g^{-1}\}$ with $ord(g)=2,3,$ and there is a part $P_j$ such that $P_i\ast P_j \subset P_i \cup P_j \cup P_k$ for different indices $i,j,k \in \{1,2,3,4\}$; \sk 
		
		\item $s=4$, $P_i=\{g\}$, $P_j=\{h\}$ and $gh, hg\in P_k$ for different indices $i,j,k \in \{1,2,3,4\}$;
	\end{enumerate}
	then there is an interval metric $d$ of $G$ associated to $\mathcal{P}$. 
\end{prop}

\begin{proof}
 In all items ($a$)--($d$) we will obtain an explicit interval metric. 
	
	\noindent $(a)$ 
	If $G=P_0 \cup P_1 \cup P_2$, just give weights $1$ and $2$ to the parts $P_1$ and $P_2$ (indistinctly).
	There is no way that the triangle inequality fail in this case. 
	
	\noindent $(b)$ 
	Suppose that $G=P_0 \cup P_1 \cup P_2 \cup P_3$.
	Give the weight 1 to the small part. That is, we can assume that $P_1=\{g,g^{-1}\}$ and take $w(P_1)=1$. If $g$ is of order $2$ or $3$ (i.e.\@ $g=g^{-1}$ or $g^2=g^{-1}$) then we give weight $2$ and $3$ to $P_2$ and $P_3$ indistinctly. Now, if $g$ is of order $\ge 4$, then we can assume that $g^2 \in P_2$. Thus, $w(g^2) \le 2w(g)\le 2$ and hence we can take $w(P_2)=2$ and $w(P_3)=3$. 
	
 	\noindent $(c)$ Give weight 1 to $P_i$. Since $ord(g)=2$ or $3$, we have $g^{2}\in P_i$, so we can we give weight 2 to $P_j$. 
 	If $h\in P_j$, by the assumption we have $w(gh)\le 3$ and hence $P_k$ must have weight 3. Now, we can give weight 4 to the remaining part $P_\ell$ without contradicting the triangle inequality. 
 	
 	\noindent $(d)$ This is a particular case of $(c)$, for groups with at least two elements of order 2. Give weight 1 and 2  to $P_i$ and $P_j$. Since $gh$ and $hg$ are in the same part, say $P_k$, we can give weight 3 to $P_k$ and weight 4 to the remaining part $P_\ell$ without violating the triangle inequality. 
\end{proof}

In the remainder of the section we will consider interval metrics on finite groups.
We first need to recall some notation from \cite{PV2}.
Given a finite group $G$ of order $n$ we define the number 
\begin{equation} \label{kG} 
k(G) = \tfrac 12 \big( n+ \#\{ x\in G : x^2 = e \} \big) -1.  
\end{equation}
It is relevant to us since the number of invariant metrics of $G$ is given by $B_{k(G)}$ where $B_m$ is the $m$-th Bell number (see Proposition 5.1 in \cite{PV2}). 
Notice that if $\mathcal{P}(G) = \{ \{e\},P_1,\ldots, P_s\}$ is the \textsl{finest} unitary symmetric partition of $G$ 
(that will be called the Lee partition $\mathcal{P}_{Lee}$ later, see Definition \ref{Lee part}), then 
\begin{equation} \label{s=kG} 
s=k(G).
\end{equation}
That is, $k(G)$ is the number of non-trivial parts of the finest unitary symmetric partition of $G$.

In particular, we want to understand interval metrics on groups $G$ of order $\le 9$. We begin by showing that if $G$ is of order $\le 7$ or $G=\Q_8$, all the invariant metrics are $\mathcal{P}$-interval.

\begin{exam}[\textit{Groups of order $\le 7$}] \label{examG<8}
	For every group $G$ of order $\le 7$ each invariant metric of $G$ is $\mathcal{P}$-interval. One can check that for any such groups the number of non-trivial parts of the unitary symmetric partitions is $s\le 4$. 
	In Examples 4.1--4.6 in \cite{PV2} one can find all these partitions explicitly. 
	For $G=\Z_2, \Z_3, \Z_4$ or $\Z_5$ we have $s\le 2$ and hence by $(a)$ in Proposition~\ref{2-weights} the metrics associated to these partitions are $\mathcal{P}$-interval.
	For $G=\Z_2 \times \Z_2$, $\Z_6$ or $\Z_7$ there are 5 unitary symmetric partitions with $s\le 3$ and only one with $s=3$ in each case. Namely, 
	\begin{gather*}
	\mathcal{P}_{\Z_2\times \Z_2} = \{ \{(0,0)\}, \{(1,0)\},\{(0,1)\},\{(1,1)\} \}, \\ 
	\mathcal{P}_{\Z_6} =\{ \{0\}, \{1,5\},\{2,4\},\{3\} \} \quad \text{and} 
	\quad \mathcal{P}_{\Z_7} =\{ \{0\}, \{1,6\},\{2,5\},\{3,4\} \}.	
	\end{gather*}
	In each case there is at least one small part, and hence by $(b)$ in Proposition~\ref{2-weights} the metrics associated to these partitions are $\mathcal{P}$-interval.
	Finally, $\Sym_3$ has 15 unitary symmetric partitions with $s \le 4$, one with $s=1$, seven with $s=2$, six with $s=3$ (with at least one small part), hence ($a$) and ($b$) of Proposition \ref{2-weights} applies, and 
	only one with $s=4$, namely 
$$\mathcal{P}=\{ \{id\}, \{(12)\}, \{(13)\}, \{(23)\}, \{(123), (132)\}\}.$$ 
Note that if $P_i=\{(12)\}$ and $P_j=\{(13)\}$ then both $(12)(13)=(132)$ and $(13)(12)=(123)$ belong to $P_k=\{(123), (132)\}$ and thus $(d)$ in Proposition~\ref{2-weights} holds. Therefore, all the metrics associated to these partitions are $\mathcal{P}$-interval for all the groups considered. 
	\hfill $\lozenge$
\end{exam}

\begin{exam}[\textit{Quaternion group}] \label{examQ8}
Consider the group of quaternions $\Q_8=\{\pm 1, \pm i, \pm j, \pm k\}$, 
with the relations $i^2=j^2=k^2=ijk=-1$. All the invariant metrics of $\Q_8$ are $\mathcal{P}$-interval. 
The finest unitary symmetric partition of $\Q_8$ is 
$$\mathcal{P}_{Lee} = \{ \{1\}, \{-1\}, \{i,-i\}, \{j,-j\}, \{k,-k\} \}.$$ 
Thus, $k(\Q_8)=4$ and any other unitary symmetric partition 
$\{ \{e\},P_1,\ldots, P_s\}$ 
has $s\le 3$. By Proposition \ref{2-weights}, all the metrics associated to a partition different from $\mathcal{P}_{Lee}$ are $\mathcal{P}$-interval (automatic if $s\le 2$ and one checks that there is at least one small part in each partition with $s=3$).
It remains to show that the metric associated to the Lee partition $\mathcal{P}_{Lee}$ (with $s=4$) is also $\mathcal{P}$-interval. 
In this case, we can use Proposition \ref{2-weights} $(c)$, with $P_i=\{-1\}$ and $P_j$ equal to any of the other parts, 
for instance $P_j=\{-j,j\}$.
\hfill $\lozenge$
\end{exam}

Next we show that $\D_4$ has some non-interval metrics and that $\Z_3\times \Z_3$ has only one non-interval metric.
We notice that we would not be able to use Proposition \ref{prop noIIM}.

\begin{exam} \label{examD4}
Consider the dihedral group of order 4, namely $\D_4 = \{e, \rho, \rho^2, \rho^3, \tau, \rho\tau, \rho^2\tau, \rho^3\tau \}$ where $\rho^4=\tau^2=e$ and $\tau \rho \tau = \rho^3$.
For $0\le i \le 3$, the elements $\rho^i$ are the rotations and $\rho^i \tau$ are the reflections. Let us see that $\D_4$ has some non-interval metrics. 
For instance, the unitary symmetric partition 
$$\mathcal{P} = \big\{ \{e\}, \{\rho,\rho^2,\rho^3\}, \{\tau\}, \{\rho\tau\}, \{\rho^2\tau\}, \{\rho^3\tau\} \big\}$$ 
does not admit an interval metric. 
Note that we cannot apply Proposition \ref{prop noIIM} since either $P_i\ast P_i=P_0$ or $P_i\ast P_i=P_0\cup P_i$ for any part $P_i$ of $\mathcal{P}$.
We prove this by considering two cases: the part $\{\rho,\rho^2,\rho^3\}$ has weight $\le 2$ or has weight $\ge 3$. 
In fact, if $\{\rho,\rho^2,\rho^3\}$ has weight $1$ or $2$, then $\{\rho^t \tau\}$ has weight $2$ or $1$ for some $0\le t \le 3$, respectively. Then, by the triangle inequality, every element would have weight less than or equal to $3$, since 
$$\{\rho,\rho^2,\rho^3\} \cdot \{\rho^t \tau\} = \{\rho^s \tau \} \qquad \text{with $s\in \{0,1,2,3\} \smallsetminus \{t\}$}.$$ 
But this is absurd since there must be $5$ non-zero values for the weights. In the other case, that is if $w(\{\rho,\rho^2,\rho^3\})\ge 3$, there are $2$ involutions with weights $1$ and $2$. Since the product of any two involutions is a rotation, this implies that $\{\rho,\rho^2,\rho^3\}$ must have weight $3$, and the previous reasoning leads us to conclude that every element should have weight less than or equal to $4$.
\hfill $\lozenge$
\end{exam}

\begin{exam} \label{z3xz3}
	The group $\Z_3 \times \Z_3$ has only one non-interval metric, which is the one associated to the finest unitary symmetric partition. 
	Let $P_0,P_1,\ldots,P_s$ be any unitary symmetric partition of $\Z_3 \times \Z_3$. 
	The finest unitary symmetric partition of $\Z_3 \times \Z_3$ is given by 
	$$ \mathcal{P}_{Lee} = \big\{ \{(0,0)\}, \{(1,0),(2,0)\}, \{(0,1),(0,2)\}, \{(1,1),(2,2)\}, \{(2,1),(1,2)\} \big\},$$
	and hence $s\le 4$.  
	If $s \le 3$ then the metric is interval by Proposition~\ref{2-weights}. Notice that in the case $s=3$ the non-trivial parts are necessary of the form $\{ g_1, g_1^{-1} \}$, $\{ g_2, g_2^{-1}\}$ and $\{g_3,g_3^{-1},g_4,g_4^{-1} \}$.
	If $s=4$, since $P_i\ast P_i=P_0 \cup P_i$ for every part in $\mathcal{P}_{Lee}$, then Proposition \ref{prop noIIM} does not apply in this case. Thus, consider a metric $d$ associated with $\mathcal{P}_{Lee}$. 
	Assume that this metric is interval and let $w$ be the associated weight of $d$. 
	Thus, there are some non-zero elements $(a,b), (c,d) \in \Z_3 \times \Z_3$ with $(a,b) \ne \pm (c,d)$ such that $w((a,b))=w((-a,-b))=1$ and $w((c,d))=w((-c,-d))=2$.
	Therefore, 
	$$w((a+c,b+d)) = w((a,b) + (c,d))\le w((a,b))+w((c,d)) \le 3.$$ 
	However, all the possible values for this weight lead to some contradiction.	
	\hfill $\lozenge$
\end{exam}

After the examples it remains to study the abelian groups $\Z_8$, $\Z_4\times \Z_2$, $\Z_2^3$ and $\Z_9$.
It is easy to see that for certain groups, like $\Z_2^{3}$, neither Proposition \ref{prop noIIM} nor Proposition \ref{2-weights} apply, so we need some more results to determine the existence or non-existence of $\mathcal{P}$-interval metrics.

We now give a generalization of Example \ref{examD4}. Recall the number $k(G)$ from \eqref{kG}.

\begin{thm} \label{kGkH}
If a finite group $G$ has a subgroup $H$ of index $2$ and $k(G)-k(H)\geq 4$ then $G$ has a non $\Prep$-interval metric. 
\end{thm}

\begin{proof}
Note that $G \smallsetminus H$ is not a subgroup, but it has symmetric partitions (necessarily not unitary) since if $g \in G \smallsetminus H$ then $g^{-1} \in G \smallsetminus H$ also. Thus, we can consider the unitary partition 
$$\mathcal{P}_{G} = \{ \{e\}, H \smallsetminus\{e\}, P_1, \ldots, P_s \}$$ 
where $P_1, \ldots, P_s$ are the non-trivial parts of the finest symmetric partition of $G\smallsetminus H$, say $\mathcal{P}_{G \smallsetminus H} = \{ P_1, \ldots, P_s \}$. Notice that $\mathcal{P}_{G}$ is symmetric since $H$ is a subgroup. Also, by  \eqref{s=kG} and hypothesis, since $H$ is not trivial we have that 
$$s = k(G)-1 \ge k(G) - k(H) \ge 4.$$
We will show that any metric $d$ associated to $\mathcal{P}_{G}$, with weight $w$, is non-interval. 

Assume, by contradiction, that $d$ is an interval metric with weight $w$ associated to $\mathcal{P}_{G}$.
Suppose first that $w(H \smallsetminus \{e\})=1$ or $2$. Then, there exists an element $g_0 \notin H$ such that 
$w(\{g_0,g_0^{-1}\})=2$ or $1$, respectively. 
Notice that $G\smallsetminus H=gH$ (since $H$ is of index 2) and 
$$(H \smallsetminus \{e\}) \cdot \{g_0,g_0^{-1}\} \supseteq g_0H.$$
Thus, by the triangle inequality we have that $w(g) \le 3$ for every $g \in G\smallsetminus H$. 
This is a contradiction since $s\ge 4$.

Now, suppose that $w(H \smallsetminus \{e\})\ge 3$. Then, there are elements $g_1, g_2  \notin H$, such that 
$w(\{g_1,g_1^{-1}\})=1$ and $w(\{g_2,g_2^{-1}\})=2$. Since $H$ is of index 2 we have that $g_1 H=g_2^{-1}H$ and hence $g_1g_2 \in H \smallsetminus \{e\}$. 
By the triangle inequality we have that $w(g_1g_2)\le 3$ and hence $w(g_1g_2)=3$. This implies that $w(H \smallsetminus \{e\})=3$.
Finally, note that $(H \smallsetminus \{e\}) \cdot \{g_1,g_1^{-1}\} \supseteq g_1 H$ and thus $w(g_1 H)\le 4$. Since $H$ is of index 2 we have that $g_1H=G\smallsetminus H$. Therefore, the weight of any element in $G$ is $\le 4$. This is a contradiction 
since $s\ge 4$, and hence there are at least 5 parts in the partition.
\end{proof}

With this result, and using that we have studied the numbers $k(G)$ in \cite{PV2} for several groups, we can assert that some families of groups have some non-interval metrics.

\begin{coro} \label{ZnDnQnSn}
The following groups have some non $\Prep$-interval metrics:
\begin{enumerate}[$(a)$]
	\item The cyclic groups $\Z_{2n}$ for $n\ge 7$, the dihedral groups $\D_n$ for $n\ge 4$, 
	the dicyclic groups $\Q_{4n}$ for $n\ge 3$ and the symmetric groups $\Sym_n$ for $n\ge 4$. \sk 

	\item The groups $\Z_2^r$ for any $r\ge 3$ and $G\times \Z_2^r$ for any non-trivial group $G$ and $r\ge 2$. 
\end{enumerate}
\end{coro}

\begin{proof}
($a$) 
For the cyclic group $G=\Z_{2n}$ consider the subgroup $H=\Z_n$ of index 2. In Example~5.2 in \cite{PV2} we proved that $k(\Z_n)=[\tfrac n2]$, hence $k(\Z_{2n}) -k(\Z_n)\ge 4$ for any $n\ge 7$.  	

In Proposition 5.5 of \cite{PV2} we showed that $k(\D_{2k})=3k$, $k(\D_{2k+1})=3k+1$ and $k(\Q_{4n})=2n$.	
For the dihedral group $G=\D_n = \langle \rho, \tau \rangle$, where $\rho^n=\tau^2=id$, 
consider the subgroup $H=\langle \rho \rangle \simeq \Z_n$ of pure rotations of index 2. 
Hence $k(\D_n)-k(\Z_n) \ge 4$ for any $n\ge 4$.
For the dicyclic group $G=\Q_{4n}$ we take a subgroup $H$ of index 2 (if $n$ is even take $H=\Q_{2n}$ while if $n$ is odd take $H=\Z_{2n}$). In any case we have  $k(\Q_{4n}) - k(H)=n\ge 4$.

For the symmetric group $G=\Sym_n$ we take the subgroup $H=\mathbb{A}_n$ of index 2. In Proposition~5.6 in \cite{PV2} we proved that $k(\Sym_n)=\tfrac 12 n! (s_n+1)-1$ and $k(\mathbb{A}_n)=\tfrac 12 n! (a_n+\tfrac 12)-1$ where  
$$s_n = \sum_{0\le k \le [\frac n2]} \frac{1}{2^k k! (n-2k)!}$$ 
and $a_n$ is the same sum as for $s_n$ but adding only over even indices $k$. It is clear that we have $k(\Sym_n)-k(\mathbb{A}_n) =\tfrac 12 n!(s_n -a_n+\tfrac 12) \ge 4$ for $n\ge 4$, since $s_n>a_n$.

($b$) In Lemma 5.3 in \cite{PV2} we showed that $k(G\times \Z_2^r) = 2^r(k(G)+1)-1$. For the group $G\times \Z_2^r$ with $r\ge 1$ we take the index 2 subgroup $H=G\times \Z_2^{r-1}$. Thus, 
$$k(G\times \Z_2^r) - k(G\times \Z_2^{r-1}) = 2^{r-1} (k(G)+1)$$ 
and this is $\ge 4$ if $G$ is trivial and $r\ge 3$ (since $k(0)=0$) or $r\ge 2$ for any non-trivial group $G$ since $k(G)\ge 1$ in this case. 
\end{proof}

A more general consequence of Theorem \ref{kGkH} is the following one.

\begin{coro}  \label{G>16}
Let $G$ be a finite group. If $|G|\ge 16$ and $G$ has a subgroup of index $2$ then $G$ has non-interval metrics. In particular, any abelian group of even order $2n$ with $n\ge 8$ has non $\Prep$-interval metrics.
\end{coro}

\begin{proof}
If $H$ is an index 2 subgroup of $G$ and $n$ denotes the order of $G$, then $H$ has order $\tfrac n2$ and, from the expression \eqref{kG}, we have 
\begin{align*}
k(G) & = \tfrac 12 \big( n + \#\{ x\in H : x^2 = e \} + \#\{ x\in G \smallsetminus H : x^2 = e \} \big) -1 \\
	 & = k(H) + \tfrac n4 + \tfrac 12 \#\{ x\in G \smallsetminus H : x^2 = e \} \ge k(H) + \tfrac n4.
	\end{align*}
Hence, $k(G)-k(H) \ge \tfrac n4$, from where the statement follows by Theorem \ref{kGkH} for any $n \ge 16$.
The remaining assertion is clear.
\end{proof}
	
Note that this corollary explains almost all cases in Corollary \ref{ZnDnQnSn} except for $\Z_{14}$, $\D_4$, $\D_5$, $\D_6$, $\D_7$,  $\Q_{12}$, $\Z_2^3$, $\Z_3 \times \Z_2^2$ and $G\times \Z_2$ with $G=\Z_3, \Z_4, \Z_5, \Z_6, \D_3$ and $\Z_7$.

\begin{exam}
As non-abelian examples of groups of even order, consider the quasidihedral groups $Q\D_n^\pm$ of order $2^n$ 
(also called semidihedral and denoted $S\D_{2^n}^{\pm}$) defined for every $n\ge 4$ by  
$$Q\D_n^\pm =  \langle x,y : x^{2^{n-1}}=y^{2}=e, yxy=x^{2^{n-2} \pm 1}\rangle.$$ 
The cyclic subgroup $H=\langle x \rangle$ is of order $2^{n-1}$ and hence of index 2 in both groups $Q\D_n^\pm $. In this way, by Corollary \ref{G>16}, the groups $Q\D_n^\pm$ have some non-interval metrics.  
\hfill $\lozenge$
\end{exam}

It is known that if the groups $G$ and $H$ have some interval metric then $G\times H$ has an interval metric given by the weight (see Proposition 12 in \cite{Ba95}):
\begin{equation} \label{prod weight}
w_{G\times H}(g,h) = (1 + \max_{h \in H} \{w_H(h) \}) \, w_{G}(g)  + w_{H}(h),
\end{equation}
where $w_G$ and $w_H$ are the weights in $G$ and $H$, respectively. 
Alternatively, one can also take 
$w_{G\times H}(g,h) = w_{G}(g) + (1 + \max_{g \in G}\{w_G(g)\} ) \, w_{H}(h)$.
However, not every interval metric of $G \times H$ comes from some interval metrics of $G$ and $H$, respectively.

Notice that the groups of the form $G\times \Z_2^r$ for some non-trivial group $G$ always have some non-interval metrics by ($b$) of Corollary \ref{ZnDnQnSn} if $r\ge 2$ and for $|G|\ge 8$, by Corollary \ref{G>16}, if $r=1$.

To finish the section, we complete the study of interval metrics on the groups of order $\le 9$.
\begin{rem} \label{remarquito}
	In Examples \ref{examG<8}--\ref{z3xz3} we have studied the existence or not of interval metrics on small groups $G$. We have seen that for all the groups of order $\le 7$ and $\Q_8$ all the metrics are $\mathcal{P}$-interval. We have also seen that $\D_4$ and $\Z_3 \times \Z_3$ have some non $\mathcal{P}$-interval metrics ($\Z_3^2$ has only one such metric).   	
	Now we can finish the study of $\mathcal{P}$-interval metrics for the groups of order $\le 9$. 
	First, notice that all the metrics associated to unitary symmetric partitions of $\Z_4\times \Z_2$ are $\mathcal{P}$-interval. 
	We have that $k(\Z_4 \times \Z_2)=5$. By Proposition \ref{2-weights}, all the unitary symmetric partitions of $\Z_4 \times \Z_2$ with $s\le 4$ are $\mathcal{P}$-interval (check). The one with $s=5$ (the Lee partition) is also $\mathcal{P}$-interval since in the next section we will show that gives rise to a Lee metric, which is interval by definition. 
	On the other hand, $\Z_2^3$ have some non $\mathcal{P}$-interval metrics by Corollary~\ref{ZnDnQnSn}. For the groups $\Z_8$ and $\Z_9$ the finest unitary symmetric partitions have $s=4$, and they are given by 
	$$\mathcal{P}_{\Z_8} = \{\{0\}, \{1,7\}, \{2,6\}, \{3,5\}, \{4\}\} \quad \text{and} \quad 
	\mathcal{P}_{\Z_9} = \{\{0\}, \{1,8\}, \{2,7\}, \{3,6\}, \{4,5\}\}.$$ 
	For both partitions we can use Proposition \ref{2-weights} $(c)$ to check that they are $\Prep$-interval.
 	Moreover, one can check that these partitions admit the $\mathcal{P}$-interval metrics of $\Z_n$ given by 
	$$w(\{i,n-i\})=i$$ 
	with $i=1,2,3,4$ for $n=8,9$ (these are the classic Lee metrics on cyclic groups, see \eqref{Lee metric}). 
	So, we can see by Proposition \ref{2-weights}, that all the metrics on $\Z_8$ and $\Z_9$ are $\Prep$-interval.    
\end{rem}

\section{Lee metrics on groups}
We recall that there is a classic Lee metric defined on finite cyclic groups. In $\Z_m$, its weight is given by 
$w_{Lee}(\bar y) = w_{Lee}(y \!\mod m)$ where 
\begin{equation} \label{Lee metric}
w_{Lee} (x) =\min \{x,m-x\}
\end{equation}
for $x \in \{0,1,\ldots,m-1\}$. Note that this is an interval metric taking the values $0,1, \ldots, \lfloor \tfrac m2 \rfloor $ and it is associated with the finest unitary symmetric partition 
$$\{ \{0\}, \{\pm 1 \}, \{\pm 2 \}, \ldots, \{\pm  \lfloor \tfrac m2 \rfloor \}  \}$$ 
of $\Z_m$. For instance, we can take 
$$w(\{\pm i\})=i$$ 
for any $i=0,1,\ldots,\lfloor \tfrac m2 \rfloor$. Notice that $w(\{\pm i\})=i$ for $i \in \Z$ defines a `Lee metric' on $\Z$, that corresponds  to the usual norm on $\Z$.

In this section we will generalize the Lee metric to arbitrary groups. 
We will show that some families of groups such as dihedral, quaternionic and dicyclic groups always have Lee metrics. In the next section we will study groups not allowing Lee metrics. 

If $\mathcal{P}_1$, $\mathcal{P}_2$ are partitions of $G$, then we say that
$\mathcal{P}_1$ is finer than $\mathcal{P}_2$ if for all $P \in \mathcal{P}_1$
there exists a set $Q \in \mathcal{P}_2$ such that $P \subseteq Q$. We will denote this relation by $\mathcal{P}_1 \preceq \mathcal{P}_2$.
The partial order $\preceq$ induces a lattice of partitions of $G$.
Clearly, the 1-1 correspondence $\mathcal{M}(G)/_{\sim_\Prep} \leftrightarrow \mathcal{P}(G)$
between $\mathcal{P}$-classes of right-invariant metrics of $G$ and unitary symmetric partitions $\mathcal{P}$ of $G$ (see Proposition 2.5 in \cite{PV2}) allows us to define a partial order on $\mathcal{M}(G)/_{\sim_\Prep}$ that we still denote by $\preceq$. 
We will denote by $\mathcal{L}(G)$ the \textit{lattice of (unitary symmetric) partitions} of $G$ given by $(\mathcal{M}(G)/_{\sim_\Prep}, \preceq)$. 
For instance, if $G=\Z_n$, then the Lee metric and the Hamming metric are representatives of the $\Prep$-classes corresponding to the minimum and the maximum of the lattice.

In general, for an arbitrary group $G$, the maximum in $\mathcal{L}(G)$ is associated to the Hamming metric. One could ask which  invariant metric is associated to the minimum in $\mathcal{L}(G)$, or, in other words, with the finest partition. 
In \cite{Ba95} the author considers this question (although more specifically for interval metrics) and answers it affirmatively for abelian groups. That is, one can associate an interval metric to the finest partition of an abelian group.

\begin{defi} \label{Lee part}
The finest unitary symmetric partition of a group $G$ will be called the \textit{Lee partition} of $G$ and denoted by $\mathcal{P}_{Lee}(G)$ or simply $\mathcal{P}_{Lee}$. It is given by
$$\mathcal{P}_{Lee}(G) = \big\{ \{a,a^{-1}\} : a\in G \} \big \}$$
or, more precisely, by 
$\mathcal{P}_{Lee}(G) = \big\{ \{e\}, \{a\} , \{b,b^{-1}\} : a,b\in G, a^2=e, a\ne e, b^2 \ne e \} \big \}$.
An interval metric $d$ on $G$ is said to be a \textit{Lee metric} on $G$ if its associated partition is the Lee partition
$\mathcal{P}_{Lee}(G)$. We will denote by $d_{Lee}$ any metric associated to the Lee partition.
\end{defi}
So, the usual Lee metric on $\Z_n$ (and its generalization to $\Z$) is a Lee metric with our definition in terms of partitions. 
In other words, the Lee partition of a group $G$ defines a Lee metric on $G$ if there is an interval metric associated to this partition. Note that this implies that $G$ must be countable (finite or not).
Not every group admits a Lee metric, since not every metric associated to the Lee partition is interval.
Notice that, by \eqref{s=kG}, if $G$ is finite we have 
$$\# \mathcal{P}_{Lee} = k(G)+1.$$ 
We notice that, in our previous work \cite{PV2}, we computed the number $k(G)$ for all finite groups of small order $n\le 32$. Also, we obtained explicit expressions for $k(G)$ for any abelian group, for some infinite families of non abelian groups such as dihedral, dicyclic and semidihedral groups of any order, and for some special semidirect products. For symmetric and alternating groups we got some non explicit but computable expressions (see Sections 5 and 7 of \cite{PV2}).

We now present two general results on groups admitting Lee metrics.

\begin{prop} \label{some Lee}
	If the group $G$ has a Lee metric then $G \times \Z_2^{k}$ has a Lee metric for any $k\in \N$.
\end{prop}

\begin{proof}
Let us see that we can associate an interval metric to the Lee patition $\mathcal{P}_{Lee}(G\times \Z_2^k)$ of the product $G\times \Z_2^k$. 	We will prove this by induction on $k$. 
	
For the base step, first notice that 
$$\mathcal{P}_{Lee}(G \times \Z_2) = \mathcal{P}_{Lee}(G) \times \mathcal{P}_{Lee}(\Z_2),$$ 
since each part of the Lee partition $\mathcal{P}_{Lee}$ of $G\times \Z_2$ is of the form 
$$\{(g,h),(g^{-1},h)\} = \{g,g^{-1}\} \times \{h\}$$ 
with $g\in G, h\in \Z_2$. 
Now let $w_G$ and $w_{\Z_2}$ be the interval weights associated to the Lee metrics on $G$ and $\Z_2$.  
Then, by \eqref{prod weight} with $H=\Z_2$, a weight on $G\times \Z_2$ is given by 
	$$w_{G\times \Z_2}((g,h)) =  2 w_G(g) + w_H(h).$$
	It is clearly an interval weight on $G\times \Z_2$ since $w_H$ takes the values $0,1$ and hence $w_{G\times \Z_2}$ takes the values $0,1,\ldots, 2k(G)+1$. 
	Since $k(G\times\Z_2)=2k(G)+1$ by Lemma 5.1 in \cite{PV2}, we can deduce that the partition associated with $w_{G\times \Z_2}$ has maximum size and thus corresponds to $\mathcal{P}_{Lee}(G \times \Z_2)$, hence $G\times \Z_2$ has a Lee metric.
	 Now, to define a Lee metric on $G\times \Z_2^k$ we use induction on the groups $G \times \Z_2^{k-1}$ and $\Z_2$.
\end{proof}

We know that a cyclic group $G$, finite or not, always has a Lee metric. If $G$ has odd order then the converse also holds. 
\begin{thm} \label{odd order}
	Let $G$ be a finite group of odd order. 
	Then, $G$ has a Lee metric if and only if $G$ is cyclic.
\end{thm}

\begin{proof}
	We only have to prove that if $G$ has a Lee metric then it must be a cyclic group. 
	
	Assume, by contradiction, that $G$ is not cyclic. 
	Let $\Prep=\{\{e\},\{a,a^{-1}\},\ldots\}$ be the Lee partition of $G$ and suppose that there exists an interval metric $w$ on $G$ associated to $\Prep$.
	Without loss of generality we can assume that $w(a)=w(a^{-1})=1$. Consider the set 
	$$A = \{ x\in G \,:\, x\not\in \langle a \rangle \} = G\smallsetminus \langle a \rangle,$$ 
	which is not empty since $G$ is not cyclic. 
	Let $b\in A$ be such that $w(b)=k$ is the minimum weight in $A$. 
	By the triangle inequality, the elements $ab,ba,a^{-1}b,b^{-1}a^{-1}\in A$ satisfy 
	$$k\leq w(ab),w(ba),w(a^{-1}b),w(b^{-1}a^{-1})\leq k+1.$$ 
	Notice that the four elements $ab, ba, a^{-1}b, b^{-1}a^{-1}$ are also different from $b$ and $b^{-1}$, otherwise we would have that $a=e$ or $b^{\pm 2}=a$, and the last identity is not possible since this would imply that $\langle b \rangle = \langle a \rangle$ as $b$ has odd order. Thus, since the part $\{b,b^{-1}\}$ has weight $k$ we must have that 
	$$w(ab)=w(ba)=w(a^{-1}b)=w(b^{-1}a^{-1})= k+1.$$ 
Note that at most two elements have the same weight and one of them is the inverse of the other. 
	
	Now, if $ab=ba$, first notice that since there are no elements of order $2$ then $ab\not = b^{-1}a^{-1}$ and $ab\not = a^{-1}b$, so the only possibility is that $a^{-1}b=b^{-1}a^{-1}$, but since $a$ and $b$ commute then we obtain $b=b^{-1}$. Otherwise, if $ab\not =ba$ then we must have that $a^{-1}b=ba$, but then $a^{-1}=bab^{-1}$ and then $a$ is conjugate to its inverse, which is well-known that cannot happen in a group of odd order. 
 In both cases we reach a contradiction and hence $G$ must be cyclic.
\end{proof}

\begin{rem} \label{fqs}
Let $\ff_q$ be a finite field of $q$ elements, hence $q=p^n$ with $p$ prime and $n\in \N$. The additive group of $\ff_q$ is isomorphic to $\Z_p^n$. If $p=2$, Proposition \ref{some Lee} implies (taking $G$ to be trivial) that $\ff_{2^n}$ has Lee metrics. On the other hand, if $p$ is odd, Theorem \ref{odd order} says that $\ff_{p^n}$ has Lee metrics if and only if $n=1$ (i.e.\@ it is $\Z_p$). 
\end{rem}

Now, we characterize all the groups having bi-invariant Lee metrics, that is, those groups in which any given (invariant) Lee metric is also bi-invariant.
\begin{prop} \label{charact bi}
Let $G$ be a group with a Lee metric $d_{Lee}$. Then, $d_{Lee}$ is bi-invariant if and only if $G$ is abelian or $G= \Q_8 \times \Z_2^k$ for some $k\in \N_0$.
\end{prop}

\begin{proof}
	By Theorem 6.8 in \cite{PV2}, every invariant metric on $G$ is also bi-invariant if and only if $G$ is abelian or 
	$G= \Q_8 \times \Z_2^k$. Thus, if $G$ is abelian or $G= \Q_8 \times \Z_2^k$ for some $k\in \N$, then every metric on $G$ is bi-invariant, in particular the Lee metric. Now, suppose that $d_{Lee}$ is bi-invariant. Then, the associated Lee partition is conjugate. Therefore, the partition associated to any other metric $d$ is also conjugate, that is to say $d$ is bi-invariant.   
\end{proof}

Notice that the groups $G=\Q_8 \times \Z_2^{k}$ have Lee metrics for any $k\in \N$, by Example \ref{examQ8} and 
Proposition~\ref{some Lee}.

\subsubsection*{Families of groups admitting Lee metrics}
Here we consider some families of metacyclic groups (groups $G$ having a normal cyclic group $H$ such that $G/H$ is also cyclic) which admit Lee metrics. Metacyclic groups include for instance cyclic groups, dicyclic groups and direct products or semidirect products of cyclic groups (in particular dihedral and quasidihedral groups). We already know that cyclic groups admit Lee metrics.

We will now show that some families of non-abelian groups do have Lee metrics while other families do not.
We will make use of a specific metric on groups that we now recall. Let $S$ be a set of generators of the group $G$ and let 
$$v:S \rightarrow \N$$ 
be any function. The \textit{weighted word metric} on $G$ associated to $S$ and $v$ is defined as 
$$d_{S,v}(x,y) = w_{S,v}(xy^{-1})$$ 
for every $x,y \in G$ where $w_{S,v}(e)=0$ and 
\begin{equation} \label{min weight}
w_{S,v}(x) = \min \Big \{ \sum\nolimits_{i=1}^k v(g_i) : x= g_1^{\varepsilon_1} \cdots g_k^{\varepsilon_k}, g_i\in S, \varepsilon_i \in \{\pm 1\}, k\in \N \Big\}.
\end{equation}
If $x= g_1^{\varepsilon_1} \cdots g_k^{\varepsilon_k}$ we will call $g_1^{\varepsilon_1} \cdots g_k^{\varepsilon_k}$ a word expression for $x$ and $v(g_1) +\cdots +v(g_k)$ the weight of the word.

In general it is difficult to decide whether this metric is interval or not.  
Our strategy will be to define interval weights $w$ on the group $G$ and check that they coincide with the weight associated to some word metric $d_{S,v}$ on $G$, and hence $w$ defines an interval metric.

First, we show that we can always define a Lee metric on dihedral groups.
\begin{thm} \label{Dn Lee}
	The dihedral group $\D_n$ has a Lee metric for every $n\ge 1$. 
\end{thm}

\begin{proof}
	Since $\D_1=\Z_2$ and $\D_2=\Z_2^2$ they have Lee metrics ($\Z_2$ is cyclic and for $\Z_2^2$ we use Proposition \ref{some Lee}). As these groups are abelian their Lee metrics are bi-invariant. So, we consider $n\ge 3$. 
	
	Instead of the classical presentation 
	$\D_n = \langle \rho, \tau \,:\, \rho^n = \tau^2 = (\tau\rho)^2=e\rangle$, 
	we will use another common presentation of $\D_n$ generated by $2$ reflections, 
	$$\D_n = \langle s,t \,:\, s^2=t^2 = e, (st)^n = e\rangle.$$ 
One can pass from one presentation to the other via the identifications $t=\tau$ and $s=\rho\tau$. 
	
Consider the weighted word metric $d_{S,v}$ on $\D_n$ associated to the generator set $S=\{s,t\}$, with $v(t)=1$ and $v(s)=2$ and let 
$w=w_{S,v}$ 
be the weight associated to $d_{S,v}$.
Now, notice that the elements of the group can be written as $e,t,s,ts,st,tst,sts,stst,tsts,\ldots$, 
that is, all the words of length up to $\lfloor n\rfloor$ formed by alternate concatenation of $s$ and $t$; namely 
$$\D_n = \big\{ e, (st)^is, (ts)^it, (st)^j, (ts)^j : 
0 \le i \le \lfloor \tfrac{n-1}2 \rfloor, 0 \le  j \le \lfloor \tfrac{n}2 \rfloor \big\}.$$ 
The words of odd length are the reflections, since when squaring every such word the $s$'s and $t$'s all cancel out. Namely, 
$$((st)^is)^2=(st\cdots sts)(st\cdots sts)=e,$$ 
and similarly for $(ts)^it$. The words of even length are the rotations, since we have 
$st=\rho \tau^2=\rho$ and $ts=\tau \rho \tau=\rho^{-1}$ and hence $(st)^k=\rho^k$ and $(ts)^k=\rho^{-k}$ for every $k$. Also, there is an additional identification given by the relation $(st)^{n}=(ts)^{n}=e$. This gives the extra involution $(st)^{\frac n2}= (ts)^{\frac n2}=\rho^{\frac n2}$ for $n$ even and $(st)^{\frac{n-1}2}s= (ts)^{\frac{n-1}2}t$ for $n$ odd.
	
	Now, for the rotations we have 
	$$w((st)^k)=w((ts)^k)=3k$$ 
	for $0\leq  k\leq\lfloor \frac{n}{2}\rfloor$, while for the reflections we have 
	$$w((ts)^kt)=3k+1 \qquad \text{and} \qquad w((st)^ks)=3k+2,$$ 
	for $0\leq  k\leq\lfloor \frac{n-1}{2}\rfloor$. 
	That is, $w$ takes all the values from $0$ to $3m$ if $n=2m$ and all the values from 0 to $3m+1$ if $n=2m+1$. Hence, $d_{S,v}$ is interval. Furthermore, the unitary symmetric partition $\mathcal{P}_w$ induced by $w$ has the same size $s$ as the Lee partition $\mathcal{P}_{Lee}$ of $\D_n$, namely $s=3m$ if $n=2m$ (resp.\@ $s=3m+1$ if $n=2m+1$). So, these partitions must coincide, that is $\mathcal{P}_{w} = \mathcal{P}_{Lee}$. 
Therefore, $d_{S,v}$ is a Lee metric on $\D_n$.
\end{proof}

\begin{exam}[Lee metric on $\D_4$]
From the proof of the previous proposition, we can give an explicit Lee metric on $\D_4=\{e,s,t,ts,st,tst,sts,tsts=stst\}$ with $s^2=t^2=(ts)^4=e$.
This metric is given by the following weights:
$$\begin{tabular}{|c|c|c|c|c|c|c|c|}
\hline
 $e$ & $t$ & $s$ & $ts$ & $st$ &  $tst$ &  $sts$ & $tsts$   \\ \hline
  $0$ & $1$ & $2$ & $3$ & $3$ & $4$ & $5$ & $6$  \\   
\hline 
\end{tabular} \qquad or \qquad 
\begin{tabular}{|c|c|c|c|c|c|c|c|}
\hline
$e$ & $\tau$ & $\rho\tau$ & $\rho^3$ & $\rho$ &  $\rho^3\tau$ &  $\rho^2\tau$ & $\rho^2$   \\ \hline
$0$ & $1$ & $2$ & $3$ & $3$ & $4$ & $5$ & $6$  \\   
\hline 
\end{tabular}$$

if one uses the presentation $\D_4=\{e, \rho, \rho^2, \rho^3, \tau, \rho\tau,  \rho^2\tau, \rho^3\tau^2\}$ with
$\rho^4=\tau^2=(\tau\rho)^2=e$. Recall from the above proof that these presentations are related by
$t=\tau$ and $s=\rho\tau$. 
\hfill $\lozenge$
\end{exam}

We now consider the family of dicyclic groups $\Q_{4n}$ (sometimes denoted $Dic_{n}$), $n\ge 1$, which includes the generalized quaternion groups $\Q_{8n}$, for any $n\ge 1$ ($\Q_8$ is just the quaternion group).

\begin{exam} 
Here we focus on the second generalized quaternion group ($\Q_8$, was previously considered) given by
	$\Q_{16} = \langle x,y \,:\, x^{8} = e, y^2 = x^4, y^{-1}xy = x^{-1}\rangle$. Notice that we can write
$$\Q_{16} = \{e,x,x^2,x^3,x^4,x^5,x^6,x^7,y,xy,x^2y,x^3y,x^4y,x^5y,x^6y,x^7y\}.$$
We consider the auxiliary function 
	$$w_1(k)=min\{2w_{Lee,8}(k), 2w_{Lee,8}(4-k)+1\},$$
	where $w_{Lee,8}$ denotes the classic Lee weight on the cyclic subgroup $\Z_8 \simeq \langle x \rangle$, and we define the following function on $\Q_{16}$:
	\begin{equation} 
	w(x^ky^j) = \begin{cases}
	w_1(k)  & \qquad \text{if } j=0, \\[1mm]
	4+ w_1(k \bmod 4) & \qquad \text{if $j=1$, } k\neq 0,4,  \\[1mm]
	5 & \qquad   \text{otherwise}. 
	\end{cases}
	\end{equation}	
It is easy to check that $w$ is an interval weight function on $\Q_{16}$. Their weights are given in the following table:	
\begin{table}[H]
$$\begin{tabular}{|c||c|c|c|c|c|c|c|c|c|c|c|c|c|c|c|c|}
\hline
 $g$ &  $e$ & $x$ & $x^2$ & $x^3$ & $x^4$ &  $x^5$ &  $x^6$ & $x^7$ & $y$ & $xy$ & $x^2y$ & $x^3y$ & $x^4y$ & $x^5y$ & $x^6y$ & $x^7y$ \\ \hline
 $w(g)$ &  $0$ & $2$ & $4$ & $3$ & $1$ & $3$ & $4$ & $2$ & $5$ & $6$ & $8$ & $7$ & $5$ & $6$ & $8$ & $7$ \\   
	\hline 
	\end{tabular}$$
\end{table}		
The associated unitary symmetric partition of this weight is given by 
$$\mathcal{P} = \big\{ \{e\}, \{x^4\}, \{x, x^7\}, \{x^3, x^5\}, \{x^2, x^6\}, \{y, x^4y\}, \{xy, x^5y\}, \{x^3y, x^7y\}, 
\{x^2y, x^6y\} \big\}.$$ 
Notice that since each part in $\mathcal{P}$ is of the form $\{g,g^{-1}\}$, the above partition is just the Lee partition. This implies that 
the distance $d$ associated to $w$ is a Lee metric on $\Q_{16}$.
\hfill $\lozenge$ 
\end{exam}

We now show that this holds for every dicyclic group. We recall that the dicyclic group $\Q_{4n}$ of order $4n$, with $n\ge 2$, has the following presentation by generators and relations  
\begin{equation} \label{quat4n}
\Q_{4n} = \langle x,y \,:\, x^{2n} = e, y^2 = x^n, y^{-1}xy = x^{-1} \rangle.
\end{equation}

This family includes the quaternion group $\Q_8$ and the groups $\Q_{8n}$ with $n\ge 2$ which are called generalized quaternions. 

\begin{thm} \label{Qn Lee}
	The dicyclic group $\Q_{4n}$ has Lee metrics for every $n\ge 2$. 
\end{thm}

\begin{proof}
Notice that every element in $\Q_{4n}$ as given in \eqref{quat4n} can be uniquely written as $x^ky^j$ with $0\leq k<2n$ and 
$0\leq j\leq 1$. Consider the auxiliary function  
	$$w_1(k) = min\{ 2w_{Lee,2n}(k), 2w_{Lee,2n}(n-k) +1 \},$$
	where $w_{Lee,2n}$ denotes the classic Lee weight on $\Z_{2n} \simeq \langle x \rangle$.
Note that $w_1$ takes all the integer values between 0 and $n$.
  
Now, we define the following function on $\Q_{4n}$:
\begin{equation}  \label{wleeQ4n}
	w(x^ky^j) = 
	\begin{cases}
	 w_1(k)  & \qquad \text{if } j=0, \\[.75mm]
	 n+ w_1(k \bmod n) & \qquad \text{if $j=1$, } k\neq 0,n,  \\[.75mm]
	 n+1 & \qquad   \text{otherwise}. 
	\end{cases}
\end{equation}
It is clear that $w$ takes all integer values from $0$ to $2n$.
We now show that this function $w$ coincides with the weighted word metric $w_{S,v}$ defined by taking 
the generator set 
	$$S=\{x,x^{-1},x^n,y,x^ny,xy,x^{n+1}y\}$$ 
and the function $v:S \rightarrow \N$ given by  
\begin{equation}
	\begin{aligned} \label{v's}
		& v(x^n)=1, \qquad & &v(x)=v(x^{-1})=2, \\ & v(y)=v(x^ny)=n+1, \qquad & &v(xy)=v(x^{n+1}y)=n+2.
	\end{aligned}
\end{equation}
We will do this in three steps by studying separately the cases $(a)$ $w(x^k)$ for $0\le k\le n$, $(b)$ $w(x^ky)$ for $k\ne 0,n$ and $(c)$ $w(y)$ or $w(x^ny)$. 

($a$) We first consider the case $j=0$ in \eqref{wleeQ4n}. Since $x^k = x\cdots x$ with the product repeated $k$-times, 
we have that $w_{S,v}(x^k) \le 2k$. Also, $x^k = x^n x^{-1} \cdots x^{-1}$ with $x^{-1}$ repeated $n-k$ times and thus
$w_{S,v}(x^k) \le 1+2(n-k)$.
One can check that for any other way of writing $x^k$ as a product of generators in $S$ we have $w_{S,v}(x^k) \ge 2k$ or $w_{S,v}(x^k) \ge  1+2(n-k)$. Hence, 
$$w(x^k)=w_1(k)=w_{S,v}(x^k).$$ 
	
($b$) Suppose now that we have an element of the form $x^k y$ with $k \ne 0,n$. In this case the element is in $\langle x\rangle y$ and hence any possible expression of $x^ky$ as a product of elements in $S$ must have at least one element in $\langle x\rangle y$. Moreover, since we are looking for the `minimum weight' expression (see \eqref{min weight}), we only need to consider words with  exactly one element in $\langle x\rangle y$, since otherwise the expression would have weight $\geq 2n$ (see \eqref{v's}).

Considering the case where the word (representing $x^ky$) has the element $y$, the expressions $(x^{k-j})(y)(x^{-j})$, have word weight 
\begin{equation}\label{case1}
w_1(k-j)+ (n+1)+w_1(-j) \geq w_1(k)+ (n+1),
\end{equation} where the last inequality is achieved for example for $j=0$. Here we have used that $w(-j)=w(j)$.

For the case that $x^{n}y$ is in the word (representing $x^ky$), the expressions $(x^{k-n-j})(x^ny)(x^{-j})$ have word weight 
\begin{equation}\label{case2}
w_1(k-n-j)+ (n+1)+w_1(-j) \geq w_1(k-n)+ (n+1),
\end{equation} where the last inequality is achieved for example for $j=0$.

If $xy$ is in the word (representing $x^ky$), the expressions $(x^{k-1-j})(xy)(x^{-j})$ have word weight 
$$w_1(k-1-j)+ (n+2)+w_1(-j) \geq w_1(k-1)+ (n+2).$$ 
Now, for $0 < k$ mod $n \le \lfloor\frac{n}{2}\rfloor$, we have that $w_1(k)=w_1(k-1)+2$. And for $\lfloor\frac{n}{2}\rfloor < k$ mod $n <n$, we have that $w_1(k-1) \le w_1(k)$, hence we obtain:
\begin{equation}\label{case3}
	w_1(k-1)+ (n+2) = 
	\begin{cases}
	 w_1(k) + n   & \quad \text{if } 0 < k \bmod n \le \lfloor\frac{n}{2}\rfloor, \\[1mm]
	  w_1(k) + n + 2  & \quad \text{if } \lfloor\frac{n}{2}\rfloor < k\bmod n <n. 
	\end{cases}
\end{equation}

If $x^{n+1}y$ is in the word (representing $x^ky$), the expressions $(x^{k-n-1-j})(x^{n+1}y)(x^{-j})$ have word weight 
$$w_1(k-n-1-j)+ (n+2)+w_1(-j) \geq w_1(k-n-1)+ (n+2).$$
For $k$ such that $0 < k$ mod $n \le \lfloor\frac{n}{2}\rfloor$ we have that $w_1(k-n-1)+ (n+2)= w_1(k-n) +n$. And for $\lfloor\frac{n}{2}\rfloor < k$ mod $n <n$, $w_1(k-1) \le w_1(k)$.
\begin{equation}\label{case4}
	w_1(k-n-1)+ (n+2) = 
	\begin{cases}
	 w_1(k-n) + n   & \quad \text{if } 0 < k \bmod n \le \lfloor\frac{n}{2}\rfloor, \\[1mm]
	  w_1(k-n) + n + 2  & \quad \text{if } \lfloor\frac{n}{2}\rfloor < k\bmod n <n.
	\end{cases}
\end{equation}

Now, comparing \eqref{case1}--\eqref{case4} for all the cases, with $k\neq 0,n$, and using that $w_1(k-n)=w_1(k)-1$ if $\lfloor\frac{n}{2} \rfloor \le k\le n + \lfloor\frac{n}{2} \rfloor$ and $w_1(k-n)=w_1(k)+1$ otherwise, we get that 
\begin{equation}
	w_{S,v}(x^ky) = 
	\begin{cases}
	 w_1(k) + n   & \qquad \text{if } 0< k \le \lfloor\frac{n}{2}\rfloor, \\[1mm]
	 w_1(k) + n   & \qquad \text{if } \lfloor\frac{n}{2}\rfloor < k < n, \\[1mm]
	 w_1(k-n) + n & \qquad \text{if } n < k \le n+\lfloor\frac{n}{2}\rfloor, \\[1mm]
	 w_1(k-n) + n & \qquad \text{if } n+\lfloor\frac{n}{2}\rfloor < k < 2n. \\
	\end{cases}
\end{equation}

Thus, we obtain that the weight of $x^{k}$ is $w_1(k)+n$ for $0 \le n$ and $w_1(k-n)+n$ for $n \le 2n$, that is $$w_{S,v}(x^{k}y)=w(x^ky)=n+w_1(k \bmod n).$$

($c$) Finally, we consider the case $j=1$ and $k=0,n$. 
Since the elements $x^ny$ and $y$ are in the coset $\langle x\rangle y$, then any possible expression as product of elements of $S$ must have at least one element in $\langle x\rangle y$. Hence $n+1$ is the minimum possible weight, since the other elements in $S\cap \langle x\rangle y$ have weight $n+2$. 

Thus, we have proved that $w(g)=w_{S,v}(g)$ for every $g \in G$, and hence $w$ is a weight function. 
	
\smallskip 
On the other hand, since $w$ is a weight function and takes all the values from $0$ to $2n$, the associated metric is interval.
Since $\Q_{4n}$ has $4n$ elements, one element of order 2 and the identity, the Lee partition of $\Q_{4n}$ is of the form $$\mathcal{P}_{Lee}= \{ \{e\}, \{x^n\}\} \cup \{ \{g,g^{-1}\} : g \in G, g\ne x^n \}$$ 
and has $2n$ non-trivial parts. 
By the relations one can check that the inverse of $a^k$ is $a^{n-k}$ and the inverse of $a^ky$ is $a^{k+n}y$, so the partition induced by $w$ is the Lee partition. 
\end{proof}

By Lemma \ref{some Lee} and Theorems \ref{Dn Lee} and \ref{Qn Lee}, 
the groups $G\times \Z_2^k$ with $G=\D_n$ or  $\Q_{4n}$ 
has Lee metrics for every $k \in \N$ and $n\ge 2$.

\section{Groups without Lee metrics}
We now study groups having no Lee metrics defined on it.
In this section we will use the following notation (as in \eqref{PiPj} for parts).
Given subsets $A,B$ of a group $G$ we write 
$$A*B = AB \cup BA = \{ab, ba: a\in A, b\in B \}.$$
Obviously if $G$ is abelian then $A*B=AB$, but we will use it for non-abelian groups.

We know from the previous section that $\Z_3 \times \Z_3$ is the smallest abelian group having no Lee metrics on it. We now consider the non-abelian case.

\begin{prop} \label{A4}
The alternating group $\mathbb{A}_4$ is the smallest non-abelian group having no Lee metrics. 
\end{prop}

\begin{proof}
Recall that the alternating group in four letters is given by 
$$\mathbb{A}_4 = \{ id, (12)(34), (13)(24), (14)(23), (123), (132), (134), (143), (124), (142), (234), (243)\}.$$
Suppose by contradiction that $\mathbb{A}_4$ has a Lee metric with associated Lee partition $\mathcal{P}=\{P_0,P_1,\ldots,P_s\}$ and thus $w(P_i)=i$ for $0\le i\le n$. Assume first that $P_1$ is of the form $\{a\}$. Then, $w(a)=1$ and $a^2=e$, i.e.\@ $a=(ij)(kl)$. Out of all the parts $P_i$ of the form $\{c,c^{-1}\}$ such that $c$ does not commute with $a$, denote by $\{b,b^{-1}\}$ the one with minimum weight $k$. Thus, $b$ must be one of the $3$-cycles since all the involutions $(ij)(kl)$ commute with each other. Therefore, one gets that $w(ab), w(ba), w(b^{-1}a) \leq k+1$
and, since $ab,ba,b^{-1}a$ are also $3$-cycles, we have that
$$w(ab)=w(ba)=w(b^{-1}a) = k+1.$$
Now, we must have that $\{ab,ba,b^{-1}a\} = \{c,c^{-1}\}$ for some element $c$. Using that $ab\neq ba$, we see that the only possibilities are $b^{-1}a=ba$ or $b^{-1}a=ab$. The first one cannot happen; for if not we would have $b=b^{-1}$, but $b$ is a 3-cycle. The second one implies that $abab=e$, so $ab$ has order $2$ which is absurd since $ab$ does not commute with $a$.

Now, assume that $P_1$ is of the form $\{a,a^{-1}\}$ with $a\ne a^{-1}$, hence $a$ is a 3-cycle.
Let $\{b,b^{-1}\}$ denote the part with minimum weight $k$ among all the parts $\{c,c^{-1}\}$ such that $ac\ne ca$. We can assume that $b$ is a $3$-cycle since otherwise we would be on the previous case. 
Since $a$ and $b$ are different 3-cycles, one can check that 
$$\{a,a^{-1}\}*\{b,b^{-1}\} = \mathbb{A}_4 \smallsetminus (\{a,a^{-1}\} \cup \{b,b^{-1}\}).$$
Thus, 
$$w(\mathbb{A}_4 \smallsetminus (\{a,a^{-1}\} \cup \{b,b^{-1}\}))= w(\{a,a^{-1}\}*\{b,b^{-1}\}) \le k+1.$$
Since $b$ has minimal weight $k$ among the elements not commuting with $a$, this means that all the 3-cycles different from $\{a,a^{-1}\}$ must have weight between $k$ and $k+1$, since they do not commute with each other. This is a contradiction, since there are $3$ parts of the form $\{c,c^{-1}\}$ with $c$ a 3-cycle different from $a$. 

Hence, we have that $\mathbb{A}_4$ does not have Lee metrics.
Since we 
have shown that $\D_n$ with $n=3,4,5,6$ and $\mathbb{Q}_8$ do have Lee metrics (see Theorem \ref{Dn Lee} and Example \ref{examQ8}), then $\mathbb{A}_4$ is the smallest non-abelian group having no Lee metrics. 
\end{proof}

\begin{rem}
	With the results given so far in Sections 2 and 3 (namely Example \ref{z3xz3}, Proposition~\ref{some Lee} and Theorems \ref{odd order}, \ref{Dn Lee}, \ref{Qn Lee}) 
	and the previous Proposition \ref{A4} we know that every group of order less than 16 has a Lee metric with the only exceptions of $\Z_3^2$ and $\mathbb{A}_4$. Also, out of the 14 groups of order 16, we can assure that all the groups have a Lee metric except for four groups: $\Z_4^2$ and the semidirect products $\Z_4 \rtimes \Z_4$, $\Z_2^2 \rtimes \Z_4$ and $(\Z_4 \times \Z_2) \rtimes \Z_4$. For these groups we cannot decide yet whether they have Lee metrics or not. 
\end{rem}

For the next result we will need the following notations. Let $G$ be a group and denote by $G_2$ and $G_{2,4}$ the subgroups of $G$ generated by all the elements of order 2 and of order 2 and 4, respectively. 
Hence, we clearly have 
$$G_2 \subset G_{2,4} \subset G.$$
Also, as usual, we denote by $Z(G)$ the center of $G$; that is, the set of elements of $G$ commuting with all the elements of $G$.

Recall, by the comments after Definition \ref{Lee part}, that an infinite group with uncountable cardinal cannot have Lee metrics. 
It is immediate from Theorem \ref{odd order} that non-cyclic groups of odd order do not have Lee metrics.  If $G$ is finite, we generalize this result to non-cyclic groups with center of odd order. Furthermore, we have the following more general result than Theorem \ref{odd order}.

\begin{thm} \label{non abelian cond}
Let $G$ be a non-cyclic group and suppose that 
$G_{2,4} \neq G$ and $Z(G)$ does not have elements of order $2$. 
Then, $G$ has no Lee metrics. 
\end{thm}

\begin{proof}
We will prove the result by contradiction. 
So, assume that $G$ has a Lee metric $d$ with associated weight $w$. In particular, $d$ is interval, and hence there is some $a\in G$ such that $w(a)=w(a^{-1})=1$. 
If $ord(a)$ is infinite, then the triangle inequality forces that $w(a^{k})=w(a^{-k})=k$ for all $k\in\mathbb{N}$ and, since $G$ is not cyclic, there is an element $b\not\in \left\langle a \right\rangle$. Now we cannot give any weight to $b$ since all the values are already assigned to the elements of $\left\langle a \right\rangle$. 

Now, if $ord(a)$ is finite, we split the proof in two cases, depending whether $a$ is in the center of $G$ or not.

\noindent ($a$) 
Suppose that $a\not\in Z(G)$. We consider the cases in which the order of $a$, $ord(a)$, is even or odd 
separately. 

$(i)$ If $ord(a)$ is odd then we define 
$$N_a = \{ x\in G \,:\, ax \ne xa \} = G \smallsetminus C(a) \ne \varnothing,$$
where $C(a) = \{g\in G : ga=ag \}$ is the centralizer of $a$ in $G$.
Clearly, $a\notin N_a$. Let $b \in N_a$ such that 
$$w(b) = w(b^{-1}) = \min_{g\in N_a} \{w(g)\} = k.$$ 

Now we have 
$$\{a,a^{-1}\} * \{b,b^{-1}\} = \{ab,ab^{-1},a^{-1}b,a^{-1}b^{-1},ba,ba^{-1},b^{-1}a,b^{-1}a^{-1}\}\subset N_a.$$ 
The inclusion holds because $b\not\in C(a)$ and then $\{a,a^{-1}\} * \{b,b^{-1}\}$ is the union of all the non-trivial cosets of $C(a)$. That is, we have
$$\{a,a^{-1}\} * \{b,b^{-1}\} = \bigcup_{g\in G \smallsetminus \{e\}} gC(a).$$ 
All the elements in $\{a,a^{-1}\} * \{b,b^{-1}\}$ must have weight $r$ with $k\leq r \leq k+1$ and are all different from $b,b^{-1}$ (if they were equal, then this would imply that either $a=e$, or $a=b^{2}$, or else $a=b^{-2}$, which forces $a$ and $b$ commute). So all the elements of $\{a,a^{-1}\} * \{b,b^{-1}\}$ have weight $k+1$ and hence $\{a,a^{-1}\} * \{b,b^{-1}\} = \{c,c^{-1}\}$ for some $c$. In this way we get that 
$$\{a,a^{-1}\} * \{b,b^{-1}\}=\{ab,ba\},$$ 
since $ab \ne ba$. In particular we have $ab=b^{-1}a=ba^{-1}$, so $b^2=a^2$, but then $\left\langle b^2\right\rangle=\left\langle a^2\right\rangle=\left\langle a\right\rangle$ (since $a$ has odd order), so $a$ and $b$ must commute, which is absurd.

($ii$) If $ord(a)$ is even, by hypothesis, there is a proper subgroup $G_{2,4}$ that contains all the elements of order $2$ and $4$ in $G$, so we can consider
$$ N_{a,4} = \{ x\in G \,:\, ax \ne xa, x\not\in G_{2,4} \} = G \smallsetminus \{C(a)\cup G_{2,4}\} \ne \varnothing.$$ 
This set is non-empty for if not one has that $G=C(a)\cup G_{2,4}$, but it is a well known fact that a group can never be the union of two proper subgroups. Thus $C(a) \cup G_{2,4}$ is a group if and only if $C(a)\subseteq G_{2,4}$ or $G_{2,4} \subseteq C(a)$, and neither can happen since $G_{2,4} \subsetneq G$ and $a\not\in Z(G)$.  
Now, as in the previous case, we obtain that $a^2=b^2$ and also that $a^2=b^{-2}$, this implies that $b^4=e$, that is $b\in G_{2,4}$ which is absurd since $b\in A$.

\smallskip
\noindent $(b)$ 
Now, assume that $a \in Z(G)$. Consider the set 
$$N_{a,2} = \{ x\in G \,:\, x\not\in\left\langle a\right\rangle, x\not\in G_2\} = G \smallsetminus 
\{\left\langle a\right\rangle \cup G_2\} \ne \varnothing.$$ 
By the same argument as in the previous case, this set is not empty because if $G_2\subseteq\langle a \rangle=G$ would imply that $G$ is cyclic, and $\langle a \rangle\subseteq G_2=G$ cannot happen by hypothesis.
Let $b \in N_a$ such that 
$$w(b)=w(b^{-1})=\min_{g\in N_{a,2}}\{w(g)\}=k.$$ 
Now, we have 
$$\{a,a^{-1}\} * \{b,b^{-1}\} = \{ab,ab^{-1},a^{-1}b,a^{-1}b^{-1},ba,ba^{-1},b^{-1}a,b^{-1}a^{-1}\} \subset N_{a,2}.$$ 
Since $Z(G)$ does not have elements of order $2$, 
then 
$a\ne a^{-1}$. Then $a^{-1}b\ne ab$, and this implies that $a^{-1}b=(ab)^{-1}$, that is $a^{-1}b=b^{-1}a^{-1}$. Now, as $a\in Z(G)$, then $b=b^{-1}$, which is absurd because $b\in N_{a,2}$.
\end{proof}

As a direct consequence of the theorem, we have the following.

\begin{coro} \label{torsion-free}
If $G$ is a torsion-free group, then $G$ has a Lee metric if and only if $G$ is cyclic.
\end{coro}
%

We next give examples of infinite groups not admitting Lee metrics. 
\begin{exam}
($i$) If $G=K \times H$ with $K$ an abelian group with center of odd order and $H$ a countable infinite non-abelian group with trivial center then, by the previous theorem, $G$ has no Lee metrics. For instance, we can take 
$$G=\Z_m \times H$$ 
with $m$ odd and 
$H=\prod_{i=1}^\infty H_i$,
where each $H_i$ is a symmetric $\Sym_n$ or an alternating group $\mathbb{A}_n$ for some $n\ge 3$, or  $H=\mathrm{PSL}(2,R)=\mathrm{SL}(2,R)/\{\pm I\}$ 
with $R=\Z$ or $\Q$. As the simplest examples we can take $G=\Z_3 \times {\rm PSL}(2,\Z)$ or $G=\Z_3 \times \Sym_3^\N$. 

\sk 
($ii$) Let $p$ be an odd prime and let $G$ be the infinite extraspecial group defined by the presentation
$$\langle z, \{x_i\}_{i\in \Z \smallsetminus \{0\}} \,:\, x_i^p=z^p=[x_i,z]=[x_i,x_j]=1 \, (|i|\ne|j|), [x_i,x_{−i}]=z \rangle.$$
Then, $Z(G)=\langle z \rangle$ has order $p$, and hence by the previous theorem $G$ has no Lee metrics. 
\hfill $\lozenge$
\end{exam}

Clearly, for a non-abelian group $G$ of odd-order we have $G_2=G_{2,4}=\varnothing$ and $Z(G)$ of odd order, and hence $G$ does not have Lee metrics. 

On the other hand, we have seen that some families of metacyclic groups (dihedral and dicyclic) do have Lee metrics. Note that in both cases one condition of the above proposition fails to hold. 
Indeed, dihedral groups $\D_{n}$ do not satisfy $G_2 \ne G$ while dicyclic groups $\Q_{4n}$ do not satisfy $G_{2,4} \ne G$.

Using the previous proposition we can list groups of order $<100$ with no Lee metrics.
Of course, there may be some other groups with no Lee metrics not covered by this proposition.  
\begin{exam} \label{ex no lee}
There are 64 groups of order $\le 100$ satisfying the hypothesis in Theorem \ref{non abelian cond}.
We give them all in Table \ref{noLee<100} listed with their GAP id number. 
For instance, for the groups $G=\mathbb{A}_4$ or $\Z_3 \times \Sym_3$ we have $G_2 =G_{2,4} \simeq \Z_2^2$, $Z(\mathbb{A}_4)=\{id\}$ and $Z(\Z_3 \times \Sym_3) \simeq \Z_3$.

{\small 
\begin{table} [h!]
		\caption{Groups with no Lee metric from Theorem \ref{non abelian cond}}  \label{noLee<100}
\begin{tabular}{||cc||cc||cc||cc||}
	\hline
group & GAP & group & GAP & group & GAP & group & GAP \\ 
\hline
$\mathbb{A}_4$            	  & (12,3) & $\Z_3 \times \D_9$    			    & (54,3) & $\Z_5 \times \D_7$   		& (70,2)  & 
$\Z_3 \times \mathrm{He}_3$   & (81,12) \\ 
$\Z_3 \times \Sym_3$ 	  	  & (18,3) & $\Sym_3 \times \Z_9$  			    & (54,4) & $\Z_3^2 \times \Z_8$ 		& (72,39) & 
$\Z_3 \times 3_-^{1+2}$  	  & (81,13) \\
$\Z_7 \rtimes \Z_3$ 	  	  & (21,1) & $\Z_3^2 \rtimes \Z_6$ 			    & (54,5) & $\Z_3 \times \Sym_4$ 		& (72,42) & 
$\Z_9 \circ \mathrm{He}_3$    & (81,14) \\
$\Z_3^2\rtimes \Z_3$  	  	  & (27,3) & $\Z_9 \rtimes \Z_6$   			    & (54,6) & $\Sym_3 \times \mathbb{A}_4$ & (72,44) & 
$\Z_7 \times \mathbb{A}_4$    & (84,10) \\
$3_{+}^{1+2}$ 			  	  & (27,4) & $\Sym_3 \times \Z_3^2$  		    & (54,12)& $\Z_5^2 \times \Z_3$ 		& (75,2)  & 
$\Z_7 \rtimes \mathbb{A}_4$   & (84,11) \\ 
$\Z_5 \times \Sym_3$  	  	  & (30,1) & $\Z_3 \times (\Z_3 \rtimes \Sym_3$)& (54,13)& $\Z_{13} \times \Z_6$      	& (78,1)  & 
$\Z_5 \times \D_9$    		  & (90,1) \\
$\Z_3 \times \D_5$    	  	  & (30,2) & $\Z_{11} \rtimes \Z_5$             & (55,1) & $\Sym_3 \times \Z_{13}$		& (78,3)  & 
$\Z_9 \times \D_5$            & (90,2) \\
$\Z_3 \cdot \mathbb{A}_4$ 	  & (36,3) & $\Z_2^3 \rtimes \Z_7$              & (56,11)& $\Z_3 \times \D_{13}$ 		& (78,4)  & 
$\Z_3^2 \times \D_5$ 		  & (90,5) \\
$\Z_3 \times\mathbb{A}_4$ 	  & (36,11)& $\Z_{19} \rtimes \Z_3$ 		    & (57,1) & $\Z_2^4 \rtimes \Z_5$		& (80,49) & 
$\Sym_3 \times \Z_{15}$       & (90,6) \\
$\Z_{13} \rtimes \Z_3$    	  & (39,1) & $\Z_3 \times (\Z_5 \rtimes \Z_4)$  & (60,6) & $\Z_3^2 \times \Z_9$			& (81,3)  & 
$\Z_{3} \times \D_{15}$	      & (90,7) \\
$\Z_7 \rtimes \Z_6$       	  & (42,1) & $\Z_5 \times \mathbb{A}_4$         & (60,9) & $\Z_9 \rtimes \Z_9$			& (81,4)  & 
$\Z_5 \times (\Z_3 \rtimes \D_{15})$   & (90,8) \\
$\Sym_3 \times \Z_7$      	  & (42,3) & $\Z_7 \rtimes \Z_9$                & (63,1) & $\Z_{27} \rtimes \Z_3$		& (81,6)  & 
$\Z_{31} \rtimes \Z_3$        & (93,1) \\
$\Z_3 \times \Z_7$        	  & (42,4) & $\Z_3 \times (\Z_7 \rtimes \Z_3)$  & (63,3) & $\Z_3 \wr \Z_3$ 				& (81,7)  & 
$\Z_2^4 \rtimes \Z_6$         & (96,70) \\
$\Z_4^2 \rtimes \Z_3$     	  & (48,3) & $\Sym_3 \times \Z_{11}$ 			& (66,1) & $\mathrm{He}_3 \cdot \Z_3$	& (81,8)  & 
$\Z_4^2 \rtimes \Z_6$  		  & (96,71) \\
$\Z_4^2 \rtimes \mathbb{A}_4$ & (48,50)& $\Z_3 \times \D_{11}$              & (66,2) & $\mathrm{He}_3 \rtimes \Z_3$ & (81,9)  & 
$\Z_2^3 \cdot \mathbb{A}_4$   & (96,72) \\
$\Z_5 \times \D_5 $ 		  & (50,3) & $\Z_7 \times \D_5$  				& (70,1) & $\Z_3 \cdot \mathrm{He}_3$   & (81,10) & 
$\Z_7 \times \D_7$ 			  & (98,3) \\ \hline
\end{tabular}
\end{table}}
In the table, $\rm{He}_3$ denotes the Heisenberg group of order 27 (for any odd prime $p$ there is a Heisenberg group, the unique non-abelian group of order $p^3$ and exponent $p$). The symbols $G \wr H$, $G \circ H$ and $G \cdot H$ denote the wreath product, the central product and the non-split group extensions of $G$ and $H$, respectively.
\hfill $\lozenge$
\end{exam}

Note that in the list of Example \ref{ex no lee} there are three affine groups of finite fields, 
$\mathrm{Aff}(\ff_7) = \Z_7 \rtimes \Z_6$, $\mathrm{Aff}(\ff_8) = \Z_2^3 \rtimes \Z_7$ and $\mathrm{Aff}(\ff_9) = \Z_3^2 \rtimes \Z_8$.
More generally, as a consequence of the previous result, we now show that the affine group 
	$$\mathrm{Aff}(\ff_q) = \ff_q \rtimes \ff_q^*$$ 
of a finite field of $q$ elements does not admit Lee metrics for any $q\ne 2,3,5$. 
In the semidirect product, $\ff_q$ stands for the additive group while $\ff_q^*$ is the multiplicative cyclic group of non-trivial elements.

\begin{prop} \label{aff}
Let $q\neq 2,3,5$ be a prime power. The affine group $\mathrm{Aff}(\ff_q)$ has no Lee metrics.
Moreover, for a subgroup of $\mathrm{Aff}(\ff_q)$ of the form $G=\ff_q \rtimes H$ with $H$ a subgroup of $\ff_q^*$ we have:
\begin{enumerate}[$(a)$]
	\item If $q$ is even, then $G$ has Lee metrics if and only if $H=\{1\}$. \smallskip

	\item Let $q$ be odd. 
	If $H=1$ then $G$ has Lee metrics if and only if $q$ is prime. If $H=\Z_2$ and $q$ is prime then $G$ has Lee metrics. 
	If $H\not\simeq \{1\},\Z_2, \Z_4$, then $G$ has no Lee metrics.
\end{enumerate}
\end{prop}

\begin{proof}
First note that since $\mathrm{Aff}(\ff_2)=\Z_2$ and $\mathrm{Aff}(\ff_3)=\Z_3 \rtimes \Z_2 = \D_3$, they have Lee metrics.
So we assume that $q\ge 4$. 
Also, it is known that affine groups $\mathrm{Aff}(\ff_q)$ are Frobenius and hence have trivial center (\cite{Kar}), and since $\mathrm{Aff}(\ff_q) \simeq \ff_q \rtimes \ff_q^*$ it has even order $q(q-1)$.

Moreover, if $H$ is trivial, the subgroup $G=\ff_q \rtimes H$ is just the finite field $\ff_q$ and the existence or not of Lee metrics in this case was established in Remark \ref{fqs}. Namely, if $q$ is even $\ff_q$ has Lee metrics, and if $q$ is odd $\ff_q$ has Lee metrics if and only if $q$ is prime. Furthermore, if $H=\Z_2$ and $q$ is prime, the subgroup $G=\ff_q \rtimes H$ is isomorphic to the dihedral group $\D_q$ (we will see below that this group is non-abelian and $\D_q$ is the only non-abelian group of order $2q$, for $q$ prime), which we know by Theorem \ref{Dn Lee} that has Lee metrics. 
So, from now we consider subgroups $G=\ff_q \rtimes H$ of the affine group $\mathrm{Aff}(\ff_q)$ with $H$ a subgroup of $\ff_q^*$ and $H \ne \{1\}$. 

The group $\mathrm{Aff}(\ff_q)$, with product 
$$(x,y)\cdot (x',y')=(xy'+y,yy'),$$ 
can be naturally seen as the subgroup of $2\times 2$ matrices of the form $(\begin{smallmatrix} a & b \\ 0 & 1 \end{smallmatrix})$ with $a \in \ff_q^*$ and $b\in \ff_q$, that is
$$\mathrm{Aff}(\ff_q) = \{ (\begin{smallmatrix} a & b \\ 0 & 1 \end{smallmatrix}) : a \in \ff_q, b\in \ff_q \}.$$ 
With this identification, every subgroup of the form $\ff_q \rtimes H$, with $H < \ff_q^*$, corresponds to 
$$\ff_q \rtimes H = \{ (\begin{smallmatrix} a & b \\ 0 & 1 \end{smallmatrix}) : a \in H, b\in \ff_q \}.$$ 
We will show that, for $H\ne \{1\}$, these groups have trivial center and they are not generated by its elements of order 2 and 4, and hence the conditions in Theorem \ref{non abelian cond} hold.

To check the condition on the center, notice that $(\begin{smallmatrix} a & b \\ 0 & 1 \end{smallmatrix})(\begin{smallmatrix} a' & b' \\ 0 & 1 \end{smallmatrix})=(\begin{smallmatrix} a' & b' \\ 0 & 1 \end{smallmatrix})(\begin{smallmatrix} a & b \\ 0 & 1 \end{smallmatrix})$ if and only if $ab'+b=a'b+b'$. Thus, if $(\begin{smallmatrix} a & b \\ 0 & 1 \end{smallmatrix})$ is in the center of $\ff_q \rtimes H$, then we see that for example taking $b'=0$ and $a'\neq 1$ (since $H\ne \{1\}$) we must have $b=0$. Then, the condition $ab'+b=a'b+b'$ becomes $ab'=b'$, and this is true for all $b'\in\F_q$ if and only if $a=1$. Thus the only matrix commuting with every other element is the identity, and hence the center of $\ff_q \rtimes H$ is trivial.

In what follows, let $G=\ff_q \rtimes H$ (if $H=\ff_q^*$ then we have the whole affine group $\mathrm{Aff}(\ff_q)$).
If $q$ is even, the elements of order 2 of $G$ are exactly the matrices of the form 
$$(\begin{smallmatrix} 1 & b \\ 0 & 1 \end{smallmatrix}).$$
On the other hand, there are no elements $(\begin{smallmatrix} a & b \\ 0 & 1 \end{smallmatrix})$ of order 4, since the square of such element would have order 2, that is $(\begin{smallmatrix} a & b \\ 0 & 1 \end{smallmatrix})^2=(\begin{smallmatrix} 1 & b \\ 0 & 1 \end{smallmatrix})$, which implies that $a^2=1$, with $a\in \ff_q^*$, but this has only the solution $a=1$. Hence, we have $G_2=G_{2,4} \ne G$. Then, we have that the elements of order 2 generate all the group $G$ if and only if $H=\{1\}$ (for instance if $q=2$). 

Now, if $q$ is odd, the elements of order 2 of $G$ are exactly the matrices of the form 
$$(\begin{smallmatrix} -1 & b \\ 0 & 1 \end{smallmatrix}).$$ 
If $q\equiv 3 \pmod 4$ there are no elements of order 4, and hence $G_2=G_{2,4}$. In this case, the elements of order 2 generates all the group if and only if $H\simeq 1$ or $\Z_2$ (for example the affine group $\mathrm{Aff}(\ff_3)$). 
On the other hand, if $q \equiv 1 \pmod 4$, any element of order 4 is of the form 
$$(\begin{smallmatrix} \pm \zeta & b \\ 0 & 1 \end{smallmatrix})$$
where $\zeta^2=-1$.
One can check that, multiplying elements of order 2 and 4 and looking at the first element of the matrix obtained, we only get elements with first matrix entry equal to $1,-1,\zeta$ or $-\zeta$. Therefore the elements of order 2 and 4 cannot generate the whole group if $H\not \simeq \{1\},\Z_2,\Z_4$. Note that this is the case for the affine group $\mathrm{Aff}(\ff_5)=\ff_5 \rtimes \Z_4$ and this is the reason why we have to exclude $q=5$ in the statement.

Thus, we see that we are in the hypothesis of Theorem \ref{non abelian cond} and hence $\mathrm{Aff}(\ff_q)$ do not have Lee metrics for any $q\neq 2,3,5$. 
\end{proof}

From the proposition, the only subgroups of the form $\ff_q \rtimes H$ of the affine groups $\mathrm{Aff}(\ff_q)$ for which we cannot decide whether they have Lee metrics or not are those with $H=\Z_2$ and $q$ not prime or $H=\Z_4$, that is the subgroups of the form $\ff_q \rtimes \Z_2$ ($q$ not prime) or $\ff_q \rtimes \Z_4$ in the case $q$ odd. See for instance the subgroups of $\mathrm{Aff}(\ff_9)$ in Table \ref{tab aff}. 

Notice that since 
$$\mathrm{Aff}(\ff_4) = (\Z_2\times \Z_2) \rtimes \Z_3 \simeq \mathbb{A}_4,$$ 
we have another proof that $\mathbb{A}_4$ has no Lee metrics. This explains in more generality the result obtained in Proposition \ref{A4}. This also follows from Theorem \ref{non abelian cond}. 
The group $\mathrm{Aff}(\ff_5)$  is excluded by Proposition \ref{aff} so we do not know, by our previous results, if it has Lee metrics or not. 
Notice that many groups appearing in Table \ref{noLee<100} in Example \ref{ex no lee} are actually affine groups $\mathrm{Aff}(\ff_q)$ or subgroups of affine groups as can be checked from the following table. In Table \ref{tab aff} below we show the first affine groups $\mathrm{Aff}(\ff_q)$ and their subgroups of the form $\ff_q \rtimes H$ with $H\ne \{1\}$ and we indicate when Proposition \ref{aff} applies.
\renewcommand{\arraystretch}{.9}
\begin{center}
	\begin{table}[h!]
		\caption{Subgroups of the affine groups $\mathrm{Aff}(\ff_q)$ for $7 \le q \le 19$} \label{tab aff}
		\begin{tabular}{ccccc}
\hline
Affine group & subgroup & GAP & Prop \ref{aff} & Lee metrics \\ \hline
$\mathrm{Aff}(\ff_7)$ & $\Z_7  \rtimes \Z_6$  & (42,1)  & \ding{51} & \ding{55} \\
					  & $\Z_7  \rtimes \Z_3$  & (21,1)  & \ding{51} & \ding{55} \\ \hline
$\mathrm{Aff}(\ff_8)$ & $\Z_2^3\rtimes \Z_7$  & (56,1)  & \ding{51} & \ding{55} \\ \hline
$\mathrm{Aff}(\ff_9)$ & $\Z_3^2\rtimes \Z_8$  & (72,39) & \ding{51} & \ding{55} \\
					  & $\Z_3^2\rtimes \Z_4$  & (36,9)  & \ding{55} & ?? \\
					  & $\Z_3^2 \rtimes \Z_2$& (18,4)  & \ding{55} & \ding{55} \\ \hline
$\mathrm{Aff}(\ff_{11})$& $\Z_{11} \rtimes \Z_{10}$  & (110,1) & \ding{51} & \ding{55} \\
					  & $\Z_{11} \rtimes \Z_5$& (55,1) & \ding{51} & \ding{55} \\
					  & $\Z_{11} \rtimes \Z_2$  & (22,1) & \ding{55} & \ding{51} \\ \hline
$\mathrm{Aff}(\ff_{13})$& $\Z_{13} \rtimes \Z_{12}$  & (156,7) & \ding{51} & \ding{55} \\
                        & $\Z_{13} \rtimes \Z_6$   & (78,1) & \ding{51} & \ding{55} \\
                        & $\Z_{13} \rtimes \Z_4$   & (52,3) & \ding{55} & ?? \\ 
					    & $\Z_{13} \rtimes \Z_3$   & (39,1) & \ding{51} & \ding{55} \\
                        & $\Z_{13} \rtimes \Z_2$   & (26,1) & \ding{55} & \ding{51} \\ \hline
$\mathrm{Aff}(\ff_{16})$& $\Z_2^4 \rtimes \Z_{15}$ & (240,191) & \ding{51} & \ding{55} \\
                        & $\Z_2^4 \rtimes \Z_5$    & (80,49)   & \ding{51} & \ding{55} \\
                        & $\Z_2^4 \rtimes \Z_3$    & (48,50)  & \ding{51} & \ding{55} \\ \hline
$\mathrm{Aff}(\ff_{17})$& $\Z_{17} \rtimes \Z_{16}$& (272,50) & \ding{51} & \ding{55} \\
						& $\Z_{17} \rtimes \Z_8$   & (136,12) & \ding{51} & \ding{55} \\
						& $\Z_{17} \rtimes \Z_4$   & (68,3)  & \ding{55} & ?? \\ 
						& $\Z_{17} \rtimes \Z_2$   & (34,1)  & \ding{55} & \ding{51} \\ \hline
$\mathrm{Aff}(\ff_{19})$& $\Z_{19} \rtimes \Z_{18}$& (342,7) & \ding{51} & \ding{55} \\
						& $\Z_{19} \rtimes \Z_9$   & (171,3) & \ding{51} & \ding{55} \\
						& $\Z_{19} \rtimes \Z_6$   & (114,1) & \ding{51} & \ding{55} \\ 
						& $\Z_{19} \rtimes \Z_3$   & (57,1)  & \ding{51} & \ding{55} \\
						& $\Z_{19} \rtimes \Z_2$   & (38,1)  & \ding{55} & \ding{51} \\ \hline
 	    \end{tabular}
	\end{table}
\end{center} 


\begin{rem}
Similar to the affine groups, we can consider the semidirect products of the form $\Z_p \rtimes \Z_q$ with $p,q$ distinct primes.
This product exists for $p=2$ or for $p$ odd with $p\equiv 1 \pmod q$.
From our previous results, we know when these groups have Lee metrics. 
First, note that $\Z_2 \rtimes \Z_q$ is the direct product $\Z_2 \times \Z_q$ and hence by Proposition \ref{some Lee} has Lee metrics.
Now, if $p$ and $q$ are both odd then $\Z_p \rtimes \Z_q$ do not have Lee metrics since the group is of odd order but not cyclic (see Theorem \ref{odd order}). Finally, if $q=2$ (hence $p$ is odd) then the group $\Z_p \rtimes \Z_2$ is the dihedral group $\D_p$, which we know that has Lee metrics.		
\end{rem}

\goodbreak

\section{Groups of small order}
In this final section, we focus on the existence of Lee metrics for groups of small order. 
More precisely, in Tables \ref{tablita}--\ref{tablitab2} we list the 93 groups of order up to 31, showing which of them have (not) Lee metrics and why (not). 
In any case, when possible, we refer to some previous results implying the existence or not of Lee metrics for each group. 
In the cases that are not covered by our results we have used direct computations with the software packages Sage \cite{Sage} and GAP \cite{GAP}.
For these calculations we need to consider all possible weight functions taking all the values from $1$ to $k(G)$ --i.e., 
every permutation of possible values from $1$ to $k(G)$-- and check if any of those weight functions define a metric on $G$ corresponding to the Lee partition. Thus, we have to check $k(G)!$ possibilities. This is not a feasible computation when the order of the group, and hence the value of $k(G)$, grows beyond the values of the tables below.

\noindent 
\textit{Comments on the tables}: In the second and third columns we list the GAP id of the groups, which
is a label that uniquely identifies the group in GAP. In the second column we also indicate between parenthesis the number of groups of the given order.

\renewcommand{\arraystretch}{.9}
\begin{table}[h!] 
	\caption{Lee metrics on groups of order $1 \le n \le 15$}  \label{tablita}
$$\begin{tabular}{|c|c|c|c|c|c|c|} 
	\hline 
\# & order  & id & $G$ 		& $k(G)$ & Lee metric	& why 	 \\	\hline
1 & $1$ (1) & 1  & $\{e\}$ 	& $0$ & \ding{51}	& trivial    \\ \hline
2 & $2$ (1) & 1  & $\Z_2$	& $1$ & \ding{51}	& cyclic     \\ \hline
3 & $3$ (1) & 1  & $\Z_3$	& $1 $  & \ding{51}	& cyclic     \\ \hline
4 & $4$ (2) & 1  & $\Z_4$	& $2$ & \ding{51}	& cyclic      \\
5 &     & 2 & $\Z_2^2$  	& $3$ & \ding{51}	& Prop.\@ \ref{some Lee}  		 \\ \hline
6 & $5$ (1)& 1 & $\Z_5$ 	& $2$ & \ding{51}	& cyclic      \\ \hline
7 & $6$ (2)& 2  & $\Z_6$	& $3$ & \ding{51}  	& cyclic    \\    
8 &     & 1  & $\D_3$   	& $4$ & \ding{51}	& Thm.\@ \ref{Dn Lee}			 \\ \hline
9 & $7$ (1)& 1  & $\Z_7$ 	& $3$ & \ding{51}	& cyclic 	   \\ \hline  
10 & $8$ (5)& 1 & $\Z_8$ 	& $4$ & \ding{51} 	& cyclic       \\  
11 &    & 2 & $\Z_4 \times \Z_2$	& $5$ & \ding{51} 	& Prop.\@ \ref{some Lee}   \\ 
12 &    & 5 & $\Z_2^3$			& $7$	 & \ding{51}	& Prop.\@ \ref{some Lee}   \\ 
13 &    & 3 & $\D_4$ 			& $6$	 & \ding{51}	& Thm.\@ \ref{Dn Lee}	   \\
14 & 	& 4 & $\Q_8$ 			& $4$	 & \ding{51}	& Exam.\@ \ref{examQ8}, Prop.\@ \ref{charact bi} 	\\ \hline
15 & $9$ (2)& 1 & $\Z_9$ 		& $4$	 & \ding{51} 	& cyclic			\\  
16 & 	& 2	& $\Z_3^2$ 			& $4$	 & \ding{55} 	& Exam.\@ \ref{z3xz3}, Thm.\@ \ref{odd order}				\\ \hline  
17 & $10$ (2)& 2 & $\Z_{10}$ 	& $5$	 & \ding{51} 	& cyclic			\\ 
18 &	  & 1 & $\D_5$ 			& $7$	 & \ding{51}  	& Thm.\@ \ref{Dn Lee} \\ \hline
19 & $11$ (1)& 1 & $\Z_{11}$ 	& $5$	 & \ding{51}	& cyclic  \\ \hline
20 & $12$ (5)& 2 & $\Z_{12}$ 	& $6$	 & \ding{51} 	& cyclic \\ 
21 &	& 5	& $\Z_3 \times \Z_2^2$ 	 & $7$ & \ding{51}  	& Prop.\@ \ref{some Lee}				  \\ 
22 &	& 4	& $\D_6$ 			& $9$	 & \ding{51}	& Thm.\@ \ref{Dn Lee}	 		   \\
23 & 	& 1	& $\Q_{12}$  		& $6$	 & \ding{51}   	& Thm.\@ \ref{Qn Lee}	 		  \\ 
24 & 	& 3 & $\mathbb{A}_{4}$ 	& $7$	 & \ding{55} 	& Prop.\@ \ref{A4}, Thm.\@ \ref{non abelian cond}	 	  \\ \hline
	25 & $13$ (1)& 1 & $\Z_{13}$ 	& $6$	 & \ding{51} 	& cyclic		   \\ \hline
26 & $14$ (2)&	2 & $\Z_{14}$ 	& $7$	 & \ding{51} 	& cyclic		   \\ 
27 & 	  &	1 & $\D_7$  		& $10$	 & \ding{51} 	& Thm.\@ \ref{Dn Lee}	 		  \\ \hline 
28 & $15$ (1)& 1 & $\Z_{15}$ 	& $7$	 & \ding{51}	& cyclic			 \\ \hline
\end{tabular}$$
\end{table}

\renewcommand{\arraystretch}{1.15}
\begin{table}[h!] 
	\caption{Lee metrics on groups of order $16 \le n \le 23$} \vspace{-4mm} \label{tablitb}
	$$\begin{tabular}{|c|c|c|c|c|c|c|} 
	\hline 
	\# & order  & id & $G$ 		& $k(G)$ & Lee metric	& why  \\	\hline
	29 & $16$ (14)& 1 & $\Z_{16}$ 	& $8$	 & \ding{51} 	& cyclic			   \\
	30 &	  & 5 & $\Z_{8} \times \Z_2$ & $9$   & \ding{51} 	& Prop.\@ \ref{some Lee}	 				  \\
	31 &	  &	2 & $\Z_4^2$ 	& $9$		 & \ding{51} &  sage			  \\
	32 &	  &	10 & $\Z_{4} \times \Z_2^2$  & $11$   & \ding{51} & Prop.\@ \ref{some Lee}					  \\
	33 &	  &	14 & $\Z_2^4$ 		& $15$	 & \ding{51} 	& Prop.\@ \ref{some Lee} 	  \\
	34 &	  & 7 & $\D_{8}$ 		& $12$	 & \ding{51} 	& Thm.\@ \ref{Dn Lee}    \\
	35 &	  & 9 & $\Q_{16}$ 		& $8$	 & \ding{51}  & Thm.\@ \ref{Qn Lee}		  \\ 
	36 &	  &	8 & $Q\D_{4}^-$ 	& $10$	 & \ding{51} & sage \\ 
	37 &	  &	6 & $Q\D_{4}^+$  	& $9$	 & \ding{51} & sage \\ 
	38 &	  &	11 & $\D_4 \times \Z_2$  & $13$ & \ding{51} & Prop.\@ \ref{some Lee}  \\
	39 & 	  &	12 & $\Q_8 \times \Z_2$ & $9$ & \ding{51} & Prop.\@ \ref{some Lee}  \\
	40 &	  & 4 & $\Z_4 \rtimes \Z_4$   & $9$	 & \ding{51} & sage  \\ 
	41 & 	  & 3 & $\Z_2^2 \rtimes \Z_4$  & $11$  & \ding{51} & sage \\
	42 &	  & 13 & $(\Z_4\times \Z_2)\rtimes \Z_2$  & $11$ & \ding{51}  & sage   \\ \hline
%
%
43& 17 (1)& 1 & $\Z_{17}$  & $8$ & \ding{51} &  cyclic \\ \hline
44 & 18 (5)& 2 & $\Z_{18}$ & $9$ & \ding{51}  & cyclic  \\
45 &       & 5 &  $\Z_3^2 \times \Z_2 $ & $9$  & \ding{55} & sage \\ 
46 &       & 1 & $ \D_9 $  & $13$  & \ding{51}   &  Thm.\@ \ref{Dn Lee} \\ 
47 &       & 3 & $ \D_3 \times \Z_3 $ & $10$  & \ding{55} & Thm.\@ \ref{non abelian cond} \\
48 &       & 4 & $ \Z_3^2 \rtimes \Z_2 $ & $13$ & \ding{55} & sage  \\ \hline
49 & 19 (1) & 1 & $ \Z_{19}$ & $9$ & \ding{51}  &  cyclic \\ \hline
50 & 20 (5) & 2 &$\Z_{20}$ & $10$ & \ding{51}  &  cyclic \\
51 & & 5 & $\Z_5 \times \Z_2^2$  & $11$ & \ding{51}  &  Prop.\@ \ref{some Lee} \\ 
52 & & 4& $\D_{10}$ & $15$ & \ding{51} &  Thm.\@ \ref{Dn Lee}   \\
53 & & 1 & $\Q_{20}$ & $10$ & \ding{51} &  Thm.\@ \ref{Qn Lee} \\ 
54 &  & 3 & $\Z_5 \rtimes \Z_4$ & $12$ & \ding{55} & sage  \\ \hline
55 & 21 (2)& 2&$ \Z_{21}$ & $10$ & \ding{51} &  cyclic \\ 
56 &  &1 & $\Z_7 \rtimes \Z_3$ & $10$ & \ding{55} & Thm.\@ \ref{odd order}, Thm.\@ \ref{non abelian cond}, Prop.\@ \ref{aff}  \\ \hline
57 & 22 (2) & 2 & $\Z_{22}$ & $11$ & \ding{51}  &  cyclic \\ 
58 &  & 1 & $\D_{11}$ & $16$ & \ding{51}  & Thm.\@ \ref{Dn Lee} \\ \hline
59 & 23 (1)& 1 & $\Z_{23}$ & $11$ & \ding{51} &  cyclic \\ \hline
\end{tabular}$$
\vspace{-8mm}
\end{table}

\renewcommand{\arraystretch}{1.15}
\begin{table}[h!] 
	\caption{Lee metrics on groups of order $24 \le n \le 31$}  \vspace{-1mm} \label{tablitab2}
	$$\begin{tabular}{|c|c|c|c|c|c|c|} 
	\hline 
	\# & order & id & $G$ 			& $k(G)$			 & Lee metric 			& why   \\	\hline
	60 & 24 (15) & 2 & $\Z_{24}$ & $12$ & \ding{51}  & cyclic    \\
	61 & & 9 & $\Z_{12} \times \Z_2$ & $13$ & \ding{51} & Prop.\@ \ref{some Lee} \\
	62 & & 15 & $\Z_3 \times \Z_2^3$ & $15$  & \ding{51} & Prop.\@ \ref{some Lee} \\ 
	63 & & 6 & $\D_{12}$ & $18$ & \ding{51}  & Thm.\@ \ref{Dn Lee}  \\
	64 & & 4 & $\Q_{24}$ & $12$ & \ding{51}  & Thm.\@ \ref{Qn Lee} \\
	65 & & 14 & $\D_6 \times \Z_2$ & $19$ & \ding{51}  & Prop.\@ \ref{some Lee}  \\
	66 & & 10 & $\D_4 \times \Z_3$ & $14$ & \ding{55}  &   sage  \\
	67 & & 5 & $\D_3 \times \Z_4$ & $15$ & \ding{51}  &   sage  \\
	68 & & 7 & $\Q_{12} \times \Z_2$ & $13$ & \ding{51}  &  Prop.\@ \ref{some Lee} \\
	69 & & 11 & $\Q_8 \times \Z_3$ & $12$ & \ding{51} &  sage \\
	70 & & 13 & $\mathbb{A}_4 \times \Z_2$ & $15$ & \ding{55}  & sage \\ 
	71 & & 12 & $\Sym_4$ & $16$ & \ding{55}  &  sage \\  
	72 & & 3 & $\mathrm{SL}_2(\ff_3)$ & $12$ & \ding{55}  &  sage  \\
	73 & & 8 & $(\Z_6 \times \Z_2) \rtimes \Z_2$ & $16$ & \ding{51} & sage \\ 
	74 & & 1 & $\Z_3 \rtimes \Z_8$ & $12$ & \ding{55} & sage   \\ \hline
%
%
	75 & 25 (2)& 1 & $\Z_{25} $ & $12$ & \ding{51} & cyclic \\ 
	76 & 	& 2 & $\Z_5^2$ & $12$ & \ding{55} & Thm.\@ \ref{odd order}\\ \hline
	77 & 26	(2)& 2& $\Z_{26}$ & $13$ & \ding{51} & cyclic  \\ 
	78 &    & 1 & $\D_{13}$ & $19$ & \ding{51} & Thm.\@ \ref{Dn Lee} \\ \hline
	79 & 	27 (5) & 1 & $\Z_{27}$ & $13$ & \ding{51} & cyclic \\
	80 & 	& 2 & $\Z_9 \times \Z_3$ & $13$ & \ding{55} & Thm.\@ \ref{odd order} \\
	81 & 	& 5 & $\Z_3^3$ & $13$ & \ding{55} &  Thm.\@ \ref{odd order} \\ 
	82 & 	& 4 & $\Z_9 \rtimes \Z_3$ & $13$ & \ding{55} & Thm.\@ \ref{odd order}, Thm.\@ \ref{non abelian cond} \\ 
	83 & 	& 3 & $\Z_3^2 \rtimes \Z_3$ & $13$ & \ding{55} & Thm.\@ \ref{odd order}, Thm.\@ \ref{non abelian cond} \\ \hline
	84 & 28	(4)& 2 & $\Z_{28}$ & $14$ & \ding{51} & cyclic \\
	85 &    & 4 & $\Z_{14} \times \Z_2$ & $15$ & \ding{51} & Prop.\@ \ref{some Lee} \\ 
	86 & 	& 3 & $\D_{14}$ & $21$ & \ding{51} & Thm.\@ \ref{Dn Lee}  \\
	87 &    & 1 & $\Q_{28}$ & $14$ & \ding{51} & Thm.\@ \ref{Qn Lee} \\ \hline
	88 & 	29 (1)& 1 & $\Z_{29}$ & $14$ & \ding{51} & cyclic\\ \hline
	89 & 30 (4)	& 4 & $\Z_{30}$ & $15$ & \ding{51} & cyclic \\
	90 & 	& 3& $\D_{15}$ & $22$ & \ding{51} & Thm.\@ \ref{Dn Lee} \\
	91 & 	& 2 & $\D_5 \times \Z_3$ & $17$ & \ding{55} & Thm.\@ \ref{non abelian cond} \\
	92 &    & 1 & $\D_3 \times \Z_5$ & $16$ & \ding{55} & Thm.\@ \ref{non abelian cond} \\ \hline
	93 & 31 (1)& 1 & $\Z_{31}$ & $15$ & \ding{51} & cyclic \\ \hline
	\end{tabular}$$
	\vspace{-5mm}
	\end{table}


\clearpage

\begin{thebibliography}{XXX}
\bibitem{Ba95} \textsc{Vladimir Batagelj}. 
\textit{Norms and distances over finite groups}. Journal of Combinatorics, Information and System Sci \textbf{20}, (1995) 243--252.

\bibitem{Ca} \textsc{Claude Carlet.} 
\textit{$\Z_{2^k}$-linear codes.} 
IEEE Trans.\@ Inf.\@ Theory \textbf{44:4}, (1998) 1543--1547.

\bibitem{HKCSS} \textsc{A.R.\@ Hammons, P.V.\@ Kumar, A.R.\@ Calderbank, N.J.A.\@ Sloane, P.\@ Solé}. 
\textit{The $\Z_4$-linearity of Kerdock, Preparata, Goethals, and related codes.} 
IEEE Trans.\@ Inf.\@ Theory \textbf{40:2}, (1994) 301--319.

\bibitem{Kar} \textsc{G.\@ Karpilovsky}. Group representations. Volume 1. Chapter 37 Frobenius group in Part B: Introduction to group representations and characters. Amsterdam: North-Holland (1992).

\bibitem{Lee58}
\textsc{C.Y.\@ Lee}.
\textit{Some properties of nonbinary error-correcting codes.} IRE Transactions on Information Theory \textbf{4:2}, (1958) 77–82.

\bibitem{Ne}
\textsc{A.A.\@ Nechaev}. 
\textit{The Kerdock code in a cyclic form}, Diskret.\@ Mat.\@ \textbf{1} (1989), 123--139. 
English translation in Discrete Math.\@ Appl.\@ \textbf{1} (1991), 365--384.

\bibitem{PV} \textsc{R.A.\@ Podest\'a, M.G.\@ Vides}. 
\textit{Isometries between finite groups}, Discrete Mathematics \textbf{343:11} (2020), 112070. 

\bibitem{PV2} \textsc{R.A.\@ Podest\'a, M.G.\@ Vides}. 
\textit{Invariant metrics on finite groups}, Discrete Mathematics \textbf{346:1} (2023), 113194. 

\bibitem{Sage} 
\textsc{W.A.\@ Stein et al.}
\textit{Sage Mathematics Software}. 
The Sage Development Team, 2020, \url{www.sagemath.org}.

\bibitem{GAP}
\textsc{The GAP Group.}
\textit{GAP -- Groups, Algorithms, and Programming,}
2019. \url{www.gap-system.org}.

\bibitem{YO} 
\textsc{B.\@ Yildiz, Z.\@ \"Odemi\c{s} \"Ozger.}
\textit{Generalization of the Lee weight to $\Z_{p^{k}}$.} 
TWMS J.\@ App.\@ Eng.\@ Math.\@ \textbf{2}, (2012) 145--153. 
\end{thebibliography}
\end{document}